\documentclass[11pt]{article}

\usepackage[margin=1in]{geometry}
\usepackage{amsmath,amssymb,amsthm,mathtools,bm}
\usepackage{microtype}
\usepackage{booktabs}
\usepackage{tabularx,array}
\usepackage{float}
\usepackage{enumitem}
\usepackage{algorithm}
\usepackage{algpseudocode}
\usepackage[numbers,sort&compress]{natbib}
\usepackage[colorlinks=true,allcolors=blue]{hyperref}
\usepackage[nameinlink,noabbrev]{cleveref}

\newtheorem{theorem}{Theorem}[section]
\newtheorem{lemma}[theorem]{Lemma}
\newtheorem{proposition}[theorem]{Proposition}

\theoremstyle{definition}

\newcommand{\R}{\mathbb R}
\newcommand{\ip}[2]{\left\langle #1,#2\right\rangle}
\newcommand{\norm}[1]{\left\lVert #1\right\rVert}
\newcommand{\grad}{\nabla}
\newcommand{\eps}{\varepsilon}

\newcommand{\cT}{\mathsf{Trial}}
\newcommand{\Scale}{\mathsf{Scale}}
\newcommand{\Radius}{\mathsf{Radius}}
\newcommand{\Success}{\mathsf{Success}}
\newcommand{\Kbar}{\overline K}
\newcolumntype{Y}{>{\raggedright\arraybackslash}X}
\newcolumntype{P}[1]{>{\raggedright\arraybackslash}p{#1}}

\title{Optimal Parameter-Free Gradient Minimization in $\ell_p$ Geometry}
\author{Shuting Yang$^*$, Yuning Yang\thanks{School of Mathematics, Guangxi University, Nanning 530004, China}~\thanks{yyang@gxu.edu.cn}}
\date{}

\begin{document}
\maketitle

\begin{abstract}
We study the first-order oracle complexity of finding a queried point with
small gradient in $\ell_p$ geometry, with particular attention to the
information needed to adapt the unknown smoothness and distance scales.  In
the strict counted local value--gradient model, no finite complexity bound can
depend only on $LR/\eps$ without a nondegenerate local scale observation: a
one-dimensional construction keeps $LR/\eps=4$ while defeating every
prescribed finite query budget.

We resolve Diakonikolas's general-$\ell_p$ parameter-free extension question
for every fixed $1<p<\infty$.  Under a nondegenerate secant initialization, the
method knows neither the smoothness constant $L$, the initial solution distance
$R$, nor $f^*$, and returns a queried point $\widehat x$ with $\|\nabla f(\widehat x)\|_q\le\eps$.  For
fixed finite $p>2$, we first establish the dimension-free deterministic
known-parameter upper exponent $p/(p+2)$ in $K=LR/\eps$, matching the
published lower polynomial exponent under its horizon and dimension
qualifications.  The finite local routine fits the same observable
scale--radius procedure, so this exponent is preserved without knowing $L$
or $R$.  Writing $\Kbar=\max\{1,LR/\eps\}$, the post-initialization
pair-oracle complexity is
$O_p(\Kbar^{1/2})$ for $1<p<2$,
$O(\Kbar^{1/2})$ for $p=2$, and
$O_p(\Kbar^{p/(p+2)})$ for $p>2$, together with the additive calibration
cost $O_p(\log(e+L/M_0))$ in every regime.
\end{abstract}

\section{Introduction}

How quickly can a first-order method make the gradient small?  For smooth
convex optimization, this poses a distinct question from minimizing the
function value.  The gradient norm is a directly observable first-order
certificate that requires no knowledge of the optimal value.  This
distinction has generated a separate
oracle-complexity theory for gradient minimization, from regularization
reductions and two-stage constructions to methods designed specifically for
small terminal gradients
\citep{Nesterov2012,NesterovGasnikovGuminovDvurechensky2021,KimFessler2021,LeeParkRyu2021,DiakonikolasWang2022}.
Most recently, Lan, Ouyang, and Zhang developed an oracle-optimal
parameter-free theory in Euclidean geometry, removing prior knowledge of the
curvature and distance scales under their secant-based initialization
\citep{LanOuyangZhang2026}.

The non-Euclidean theory developed along a different axis.  In general norms,
the geometry needed to make gradients small depends essentially on the
underlying space.  Diakonikolas subsequently asked whether the Euclidean
parameter-free result could be extended to other $\ell_p$ norms using
general-norm complementary-composite methods
\citep[p.~2]{Diakonikolas2025OpenProblems}.  We resolve this question for every fixed $1<p<\infty$.

The first issue is informational: how can curvature adaptation begin if the
oracle has not yet revealed any positive local curvature scale?  In the
Euclidean parameter-free method of
\cite{LanOuyangZhang2026}, curvature adaptation is
initialized by a nondegenerate secant observation: the algorithm selects a
point $z_0\ne x_0$ with $\nabla f(z_0)\ne\nabla f(x_0)$ and thereby obtains the secant
quotient
\[
 M_0=
 \frac{\|\nabla f(z_0)-\nabla f(x_0)\|_2}
      {\|z_0-x_0\|_2}>0
\]
in our notation.
This observation raises a key question: is obtaining such a
nondegenerate secant an innocuous initialization step, or can discovering
it itself incur a nontrivial worst-case oracle cost?

We show that this initialization need not be free.  If the local
value-gradient queries used to discover such a point are counted, then no
finite complexity bound can depend only on $LR/\eps$.  The obstruction
already appears in one dimension, even with $LR/\eps=4$: for any prescribed
finite query budget, one can construct a smooth convex function for which
every point explored by the method has the same gradient as $x_0$, while both
the first nondegenerate secant and the region of small gradients lie beyond
the explored region.  Thus the cost of producing the secant initialization
itself is not controlled by $LR/\eps$ alone.

Because the information obstruction is one-dimensional, it applies to every
interior $\ell_p$ geometry.  We work under the nondegenerate
secant-initialization convention of Lan--Ouyang--Zhang
\citep{LanOuyangZhang2026}.  Under this information convention,
\Cref{thm:main} gives the rates stated below for every fixed
$1<p<\infty$
in the smooth convex setting considered here, where each oracle call
returns an exact value--gradient pair and the output must be a queried point
\citep[p.~2]{Diakonikolas2025OpenProblems}.  Below two, the resulting
dimension-uniform upper bound has no multiplicative accuracy logarithm and
matches the exponent $1/2$ in $K=LR/\eps$ of the published deterministic
lower benchmark; that benchmark retains an $\ln d$ factor together with
horizon, dimension, and model qualifications.  The endpoints, nonconvex
extensions, randomized optimality, exact dependence on $p$, $d$, or
$\ln d$, arbitrary unqueried outputs, and bit complexity remain outside the
result.

For $1<p<2$, the known-parameter dynamics are prior work.  Kim, Park,
Ozdaglar, Diakonikolas, and Ryu proved that $N$ accelerated mirror-descent
steps followed by $N$ mirror-dual steps attain the fixed-horizon terminal
rate $O_p(LR/N^2)$ when $L$ is known
\citep[Cor.~3]{KimParkOzdaglarDiakonikolasRyu2024}.  The contribution here is
to turn those dynamics into a finite observable trial at guesses $(M,D)$:
every consecutive pair is queried before its cocoercive test, failure gives
$M<L$, and an all-passing nonsuccess gives $D<R$.  Coupling this exact
$\Success/\Scale/\Radius$ interface to the secant controller preserves the
square-root rate when both $L$ and $R$ are unknown.

For finite $p>2$, the answer includes an additional deterministic
known-parameter optimality result.  The complementary-composite construction
of Diakonikolas and Guzm\'an obtains the dimension-free upper exponent
$2(p-1)/(p+2)$, whereas their lower bound has exponent $p/(p+2)$ in
$K=LR/\eps$
\citep[Eq.~(18), Thm.~3, Secs.~3--4]{DiakonikolasGuzman2024}.  Their route
first drives a regularized function
gap to the scale required by the subsequent gradient conversion---of order
$\eps^2/M$ in the smooth trial normalization used here---and the resulting
balance gives the larger exponent.  The present two-phase construction stops
its first phase already at gap $O(\eps D)$ and then uses the anti-diagonal
mirror dual to control
$h^*(\nabla F)=\|\nabla F\|_q^q/q$ directly at a queried terminal point; the
finite anti-diagonal architecture is adapted from
\citet{KimParkOzdaglarDiakonikolasRyu2024}.  If $\eta$ denotes the permitted
accumulated error, its terminal weight satisfies
$u_N\asymp_p\eta^{(p-2)/p}N^{(p+2)/p}$ and preserves the
$N^{-(p+2)/p}$ scale, whose inversion gives the upper exponent $p/(p+2)$.
This matches the published lower polynomial exponent when the selected
lower-bound horizon fits in the dimension.  Observable tests at trial values
of $L$ and $R=\operatorname{dist}_p(x_0,X^*)$ place the finite routine inside
the same scale--radius procedure, so parameter adaptation preserves the new
exponent.

Given $x_0,z_0$ and the observable secant scale $M_0>0$, the method needs no
knowledge of $L$, $R$, or $f^*$ and returns a queried point $\widehat x$ with
$\|\nabla f(\widehat x)\|_q\le\eps$.  Writing $\Kbar=\max\{1,LR/\eps\}$,
the post-initialization complexity is
$O_p(\Kbar^{1/2})$ for $1<p<2$,
$O(\Kbar^{1/2})$ for $p=2$, and
$O_p(\Kbar^{p/(p+2)})$ for $p>2$, together with the additive
$O_p(\log(e+L/M_0))$ calibration cost.

The secant calibration and distance-scale guess-and-check retain the
parameter-free controller lineage of Lan--Ouyang--Zhang.  Each
geometry-specific local routine is organized so that a finite trial either
returns a queried point $\widehat x$ with
$\|\nabla f(\widehat x)\|_q\le\eps$, or establishes
\[
  M<L \qquad\text{or}\qquad D<R.
\]
The first implication comes from an observed failure of a finite smoothness
test.  For the second, each $p$-regime proves
\[
  D\ge R \ \text{and all smoothness tests pass}
  \quad\Longrightarrow\quad
  \exists\,\widehat x\ \text{queried with }\|\nabla f(\widehat x)\|_q\le\eps.
\]
Hence a fully checked but unsuccessful trial proves $D<R$.  The controller
therefore increases only a trial value already proved insufficient, and the
visited costs can be summed first over radii and then over curvature scales.

The conditional-success proof depends on the geometry.  For $1<p<2$, a
guarded accelerated mirror-descent phase produces a value gap, and its finite
anti-diagonal mirror dual converts that gap to the gradient at a queried
terminal point.  Both phases use only consecutive $\mathsf C_M$ tests, and
the local cost is $O_p(\sqrt{MD/\eps})$.  For $p=2$, the two-stage Euclidean
construction \citep{NesterovGasnikovGuminovDvurechensky2021,KimFessler2021}
is turned into a finite certificate at a trial scale by checking the
smooth-convex interpolation conditions of
\citet{TaylorHendrickxGlineur2017} together with a terminal observation.  For
$p>2$, consecutive queried pairs supply the oriented Bregman test needed to
run the two phases above at a trial curvature.  The resulting local cost is
$O_p((MD/\eps)^{p/(p+2)})$, and the two geometric sums preserve that
exponent.

In summary, our contributions are threefold.
\begin{enumerate}[leftmargin=2.2em]

\item
We show that, if the cost of obtaining a nondegenerate secant is counted,
no finite complexity bound depending only on $LR/\eps$ is possible, even in
one dimension with $LR/\eps=4$.  The result also covers randomized
finite-horizon methods and worst-case expected time.

\item
For fixed finite $p>2$, we establish the deterministic known-parameter
polynomial exponent $p/(p+2)$ for queried outputs.  The upper bound is
dimension-free; the matching lower polynomial exponent retains its
$T\le d$ and $\min\{p,\ln d\}$ qualifications.

\item
We resolve Diakonikolas's parameter-free extension question for every fixed
$1<p<\infty$.  Under the nondegenerate secant-initialization convention, the
method has the three complexity rates stated above.  Below two, the finite guarded realization and
its exact queried-pair chronology connect the prior known-parameter
AMD/mirror-dual dynamics to the same three observable outcomes.  Thus
adaptation to unknown $L$ and $R$ preserves the local square-root rate without
using $f^*$.

\end{enumerate}

\subsection*{Related work}

Gradient-norm minimization for smooth convex functions has been developed through several complementary constructions.  Nesterov's regularization reduction
\citep{Nesterov2012} was followed by two-stage schemes that first reduce function error and then control the terminal gradient \citep{NesterovGasnikovGuminovDvurechensky2021}, and by OGM-G, designed specifically
for terminal-gradient decrease \citep{KimFessler2021}.  A geometric view of acceleration
led to related small-gradient constructions \citep{LeeParkRyu2021}, while the
potential-function framework of Diakonikolas and Wang gives a common analysis of several convex and min--max gradient-reduction schemes \citep{DiakonikolasWang2022}.

Smooth objective minimization in Banach spaces has a classical optimal-method lineage \citep{NemirovskiiNesterov1985}.  For the terminal-gradient criterion, the oracle-complexity
setting with nonstandard input and output norms was isolated by Guzm\'an \citep{Guzman2015Nonstandard}, with high-dimensional lower-bound benchmarks in
$\ell_p$ geometry developed by Guzm\'an and Nemirovski \citep{GuzmanNemirovski2015}.
Diakonikolas--Guzm\'an then developed complementary-composite methods for small gradients in general norms \citep{DiakonikolasGuzman2024}.  Below two, their
regularization route gives an upper bound with a multiplicative accuracy logarithm,
while their displayed smooth lower benchmark is
$\Omega(\sqrt{LR/(\eps\ln d)})$ under its dimension and horizon conditions.
Above two, their rate comparison identifies the deterministic fixed-$p$
polynomial-exponent gap closed here.  Our lower construction uses the local smoothing and resisting completion of Guzm\'an--Nemirovski and the radial cap of
Diakonikolas--Guzm\'an \citep{GuzmanNemirovski2015,DiakonikolasGuzman2020}.

Our finite-horizon coefficient construction uses the mirror-duality framework of Kim, Park, Ozdaglar, Diakonikolas, and Ryu, including its CFOM
representation, anti-diagonal transform, and terminal-gradient relation
\citep{KimParkOzdaglarDiakonikolasRyu2024}.  In particular, their
Corollary~3 already gives the known-$L$, fixed-horizon $O_p(LR/N^2)$ terminal
gradient rate for $1<p\le2$.  Below two, we adapt their CFOM/AMD and
mirror-dual coefficient architecture using the local finite plateau schedule displayed in \eqref{eq:below-weights}; we then make the inequalities depending
on smoothness observable at a trial value $M$ and prove the exact
three-outcome interface needed for unknown $L,R$.  The
$p$-power geometry needed when $p>2$ instead leaves a mixed residual whose
negative part must be controlled and accumulated explicitly.  Bai and
Bullins obtain a queried-iterate $2/3$ gradient rate under a
trajectory-sensitive Euclidean-radius accounting; converting only a uniform
$\ell_p$ radius retains the factor $d^{(p-2)/(3p)}$, together with their
stated inner-search overhead \citep{BaiBullins2025}.  Thus, when only $R_p$
is controlled, this conversion retains dimension dependence, whereas the
bound proved here is dimension-free and has exponent $p/(p+2)$.

Parameter adaptation has a separate lineage.  Adaptive regularization under the gradient-norm criterion was studied by Ito and Fukuda
\citep{ItoFukuda2021}.  Universal model tests adapt to unknown H\"older smoothness \citep{Nesterov2015Universal}, restart schemes adapt objective-gap methods without
knowing the optimum \citep{RenegarGrimmer2022}, and recent line-search-free methods
obtain function-value guarantees across unknown smoothness regimes \citep{LiLan2025}.
These forms of unknown-smoothness or objective-gap adaptation do not by themselves resolve simultaneous unknown curvature and unknown solution distance under a
terminal-gradient criterion.  The Euclidean parameter-free method of
Lan--Ouyang--Zhang provides the baseline for the present scale and distance
adaptation
\citep[Sec.~4]{LanOuyangZhang2026}, which combines secant-based calibration
and backtracking with guess-and-check adaptation to the unknown distance
scale.  Diakonikolas--Guzm\'an's general-$\ell_p$ regularization, by contrast,
depends explicitly on the initial solution distance.  The issue addressed
here is how to retain the general-$\ell_p$ small-gradient rates when both
scales are unknown.  We modify the local routines so that an unsuccessful run reveals which guess is too small, avoiding an independent search over
both scales.  The oriented Bregman-to-gradient quotient used in the new
above-two trial has an adaptive line-search precedent in A2GD
\citep{XuChen2026A2GD}; the all-passing radius implication and the two
geometric sums are proved here for the present finite routine.

The information model creates a separate boundary.  Stochastic small-gradient methods and their local-versus-global oracle complexity are studied by Allen-Zhu and Foster
et al.\ \citep{AllenZhu2018,FosterEtAl2019}.  In contrast, Carmon--Hinder and
Kreisler et al.\ establish parameter-free stochastic \emph{objective-suboptimality}
guarantees \citep{CarmonHinder2022,KreislerIvgiHinderCarmon2024}; Lawrence et al.\
further separate sample complexity from gradient-query and computational complexity
when problem parameters are unknown \citep{LawrenceEtAl2026}.  Attia and Koren give
impossibility and price results for stochastic gradient-only parameter-free optimization
\citep{AttiaKoren2024}.  Our lower bound is related in spirit to these limits of adaptation,
but it uses a different information model.  Arjevani and Shamir study oblivious
first-order algorithms \citep{ArjevaniShamir2016}, and lower bounds for finding stationary
points in broader smooth classes provide a further benchmark
\citep{CarmonDuchiHinderSidford2020}.  Here the oracle returns exact local function values and gradients, and the obstruction is one-dimensional: even
with $LR/\eps=4$ fixed, a finite transcript need not reveal either a
nondegenerate secant or a region of small gradients.  Thus our negative result
concerns the cost of discovering scale information itself, rather than
limitations of obliviousness, stochastic feedback, or nonconvex stationarity.

\paragraph{Organization} The paper proceeds as follows.  \Cref{sec:models}
separates the two oracle models, proves the impossibility theorem, and states
the known-parameter above-two comparison and \Cref{thm:main} together with
its proof structure.
\Cref{sec:certification} develops the controller;
\Cref{sec:below,sec:euclidean,sec:above} prove the three local routines;
and \Cref{sec:global} sums their costs and completes the theorem proof.
\Cref{sec:conclusion} draws conclusions.

\section{Oracle models, information boundary, and main results}
\label{sec:models}

Fix $1<p<\infty$, set $q=p/(p-1)$, and let $\eps>0$ denote the target gradient norm.

\subsection{Two information conventions}

In the \emph{strict counted local model}, the procedure begins from an initial
point $x_0$, and every exact local query at $x$ returns the pair
$(f(x),\grad f(x))$.  The next query and final finite output are measurable
functions of the finite transcript (and internal randomness, if present).
There is no supplied nondegenerate secant observation and no uncounted nonlocal
line minimization.

For \Cref{thm:main}, we use a different, \emph{post-initialization secant model}.
For a differentiable objective $f:\R^d\to\R$, in addition to $x_0$ the method
receives a point $z_0\ne x_0$ with $\grad f(z_0)\ne\grad f(x_0)$ and the
observable number
\begin{equation}\label{eq:M0}
 M_0:=\frac{\norm{\grad f(z_0)-\grad f(x_0)}_q}
 {\norm{z_0-x_0}_p}>0.
\end{equation}
Under the $L$-smoothness condition imposed below, $M_0\le L$.  The cost of
discovering $z_0$ is outside the count; after it is supplied, every exact
pair-oracle call is counted.  This convention corresponds to the secant
initialization used by Lan--Ouyang--Zhang
\citep[Sec.~4]{LanOuyangZhang2026}.  In their notation, the raw secant quotient
\eqref{eq:M0} is $\widetilde M_0$; their symbol $M_0$ denotes the value obtained
after a backtracking calibration.  We retain $M_0$ for the raw primitive input
and use $M_{\rm a}$ for the accepted calibrated scale.

The information distinction, rather than the ambient norm alone, determines
whether a parameter-free conclusion is possible.  The following matrix states
the boundary used throughout the paper.

\begin{table}[ht]
\centering
\small
\caption{Information model versus attainable result.  The Euclidean entry
refers to the secant-initialized calibration convention of
Lan--Ouyang--Zhang; our primitive inputs are defined in \eqref{eq:M0}.}
\label{tab:information-matrix}
\begin{tabularx}{\textwidth}{@{}p{0.29\textwidth}YY@{}}
\toprule
Information model & Euclidean $p=2$ & General interior $\ell_p$ geometry \\
\midrule
Strict counted local oracle, with no nondegenerate secant observation
& Impossible by \Cref{thm:impossibility}
& Impossible by the same one-dimensional theorem, since every norm on
  $\R$ is absolute value \\
\addlinespace
Nondegenerate secant initialization, followed by counted local calls
& Euclidean parameter-free baseline of Lan--Ouyang--Zhang, recovered within
  \Cref{thm:main}
& \Cref{thm:main} for every fixed $1<p<\infty$ \\
\bottomrule
\end{tabularx}
\end{table}

The relevant dichotomy is \emph{no nondegenerate secant observation versus
an observable one}, across both Euclidean and non-Euclidean geometries.
The first side permits the hidden transition in \Cref{thm:impossibility}; the
second permits adaptation to unknown curvature and distance in all fixed
interior $\ell_p$ geometries.  The table records the two interfaces studied
here and makes no uniqueness claim about the secant interface.

\subsection{A strict-model information boundary}

A nondegenerate secant point may look easy to obtain for a generic nonlinear
objective.  The difficulty is uniformity: a worst-case function can hide all
gradient variation arbitrarily far from $x_0$ while keeping the normalized
smoothness--distance product fixed.  The theorem below formalizes this
obstruction.

The hard family below uses the affine--quadratic shape familiar from
performance-estimation worst cases \citep{DroriTeboulle2014,TaylorHendrickxGlineur2017},
but changes its role: the transition scale is selected after an arbitrary
exact value-gradient transcript rather than for a fixed method's objective
error.  Neighboring impossibility results concern oblivious first-order
algorithms \citep{ArjevaniShamir2016} or stochastic gradient-only information
\citep{AttiaKoren2024}; lower bounds for stationary points use different
classes and criteria \citep{CarmonDuchiHinderSidford2020}.

\begin{theorem}[Scale-identification impossibility]\label{thm:impossibility}
Fix $1<p<\infty$ and a supplied accuracy $\eps>0$.
\begin{enumerate}[label=(\alph*)]
\item For every deterministic strict local exact value-gradient method, every
  positive integer $N$, and every prescribed $x_0\in\R$, there is a convex,
  coercive $C^1$ function $f:\R\to\R$ with a unique minimizer $x^*$ and exact
  gradient-Lipschitz constant $L>0$ such that, with $R:=|x^*-x_0|$,
  $LR/\eps=4$, none of the first $N$ queries has $|f'(x)|\le\eps$, and any
  additional finite transcript-measurable output can also be forced to have
  gradient magnitude larger than $\eps$.
\item For every randomized method, finite horizon $N$, and $\delta\in(0,1)$,
  one deterministic member of the same family has $LR/\eps=4$ and success
  probability at most $\delta$ through horizon $N$, including any finite
  transcript-measurable output.
\item Consequently, the supremum over instances with $LR/\eps=4$ of the
  expected hitting time of an $\eps$-gradient point is infinite.
\end{enumerate}
The statement holds simultaneously for every interior $\ell_p/\ell_q$
geometry because all one-dimensional norms equal absolute value.
\end{theorem}

\paragraph{Proof roadmap.}
The deterministic construction first freezes the supplied accuracy and then
runs the method against an exactly affine value--gradient oracle.  Because the
resulting transcript is finite, a transition length $H$ can be selected beyond
every query and every transcript-measurable finite output.  A convex
affine--quadratic--affine function with that transition length reproduces the
entire transcript while keeping $LR/\eps$ fixed.  For a randomized method the
same argument is applied on a high-probability bounded-transcript event and
the two runs are coupled with the same seed.  Finally, the resulting tail
bound is converted directly into an expected-time lower bound.

\begin{proof}
\emph{Deterministic transcript.}
The quantifier order is important.  First fix the accuracy $\eps$ supplied to
the method, and only then put $g:=2\eps$.  Simulate the deterministic method
for $N$ calls against the affine oracle
\begin{equation}\label{eq:affine-oracle}
 \ell(x)=-g(x-x_0),\qquad \ell'(x)=-g.
\end{equation}
Write the affine transcript after $t$ calls as
\begin{equation}\label{eq:affine-transcript}
 \mathcal T_t^\ell
 :=\bigl((x_s^\ell,\ell(x_s^\ell),\ell'(x_s^\ell))\bigr)_{s=1}^t,
 \qquad \mathcal T_0^\ell:=\varnothing.
\end{equation}
Because the method is deterministic, there are transcript maps
$\mathcal A_{t+1}$ and, when the method returns an additional point, an output
map $\Phi_N$ such that
\begin{equation}\label{eq:query-output-maps}
 x_{t+1}^\ell=\mathcal A_{t+1}(\mathcal T_t^\ell),
 \qquad
 y_N^\ell=\Phi_N(\mathcal T_N^\ell).
\end{equation}
Thus the simulation produces fixed real queries
$x_1^\ell,\ldots,x_N^\ell$ and, when applicable, one further fixed real point
$y_N^\ell$.  Define
\[
 B_N:=\max\bigl(\{0\}\cup\{x_t^\ell-x_0:1\le t\le N\}
          \cup\{y_N^\ell-x_0\}\bigr),
\]
omitting the last singleton when no additional output is made.  The set is
finite, so $B_N<\infty$.  We may consequently choose $H>B_N$.  Notice that
$H$ is selected after a finite transcript but before the adversarial instance
is declared; this is exactly the order in the existential conclusion of
part~(a).

With $z=x-x_0$, define
\begin{equation}\label{eq:hard-family}
 f_H(x)=
 \begin{cases}
 -gz, & z\le H,\\[1mm]
 -gz+\dfrac{g}{2H}(z-H)^2, & H\le z\le3H,\\[3mm]
 gz-4gH, & z\ge3H.
 \end{cases}
\end{equation}
At $z=H$, the middle value equals $-gH$ and its derivative equals $-g$;
at $z=3H$, the middle value equals $-gH$ and its derivative equals $g$.
These agree with the two affine pieces.  Thus the derivative is
\begin{equation}\label{eq:hard-derivative}
 f_H'(x)=
 \begin{cases}
 -g, & z\le H,\\[1mm]
 g(z/H-2), & H\le z\le3H,\\[1mm]
 g, & z\ge3H.
 \end{cases}
\end{equation}
The displayed derivative is continuous and nondecreasing, proving that
$f_H$ is $C^1$ and convex.  The left tail satisfies
$f_H(x)=-g(x-x_0)\to+\infty$ as $x\to-\infty$, and the right tail satisfies
$f_H(x)=g(x-x_0)-4gH\to+\infty$ as $x\to+\infty$; hence $f_H$ is coercive.
The derivative vanishes only when $z/H-2=0$, so the unique minimizer is
$x^*=x_0+2H$.

The derivative is constant on the tails and has slope $g/H$ on $[H,3H]$.
Consequently it is globally $g/H$-Lipschitz.  This constant is exact, because
for any distinct $z_1,z_2\in(H,3H)$,
\[
 \frac{|f_H'(x_0+z_2)-f_H'(x_0+z_1)|}{|z_2-z_1|}=\frac gH.
\]
Thus $L=g/H$ is the exact gradient-Lipschitz constant.

\emph{Indistinguishability and normalization.}
On the whole half-line $z<H$, both the value and derivative in
\cref{eq:hard-family,eq:hard-derivative} agree exactly with
\eqref{eq:affine-oracle}.  Let
\[
 \mathcal T_t^H
 :=\bigl((x_s^H,f_H(x_s^H),f_H'(x_s^H))\bigr)_{s=1}^t,
 \qquad \mathcal T_0^H:=\varnothing
\]
be the transcript on the declared instance.  For $0\le t<N$, the induction
step is explicit:
\begin{align}
 \mathcal T_t^H=\mathcal T_t^\ell
 &\Longrightarrow
 x_{t+1}^H
 =\mathcal A_{t+1}(\mathcal T_t^H)
 =\mathcal A_{t+1}(\mathcal T_t^\ell)
 =x_{t+1}^\ell<x_0+H,\label{eq:query-coupling}\\
 &\Longrightarrow
 \bigl(f_H(x_{t+1}^H),f_H'(x_{t+1}^H)\bigr)
 =\bigl(\ell(x_{t+1}^\ell),\ell'(x_{t+1}^\ell)\bigr),\nonumber\\
 &\Longrightarrow \mathcal T_{t+1}^H=\mathcal T_{t+1}^\ell.\nonumber
\end{align}
Starting from the empty transcript gives
$\mathcal T_t^H=\mathcal T_t^\ell$ for every $t\le N$.  The output map
therefore gives
\[
 y_N^H=\Phi_N(\mathcal T_N^H)
      =\Phi_N(\mathcal T_N^\ell)=y_N^\ell<x_0+H
\]
whenever an additional output is present.

Every queried or returned point therefore has gradient magnitude
$g=2\eps>\eps$.  At the same time
\[
 R=2H,\qquad \frac{LR}{\eps}
 =\frac{(g/H)(2H)}{g/2}=4.
\]
This proves part~(a), including the known-$\eps$ quantifier order and an
arbitrary finite transcript-measurable output.

\emph{Randomized finite horizon.}
Let $\omega$ denote the method's random seed and run it against the affine
oracle.  Write $\mathcal T_t^{\ell,\omega}$ for the resulting transcript and
let $M_N(\omega)$ be the maximum of zero and all positive displacements among
the first $N$ finite queries and any finite output.  Since this is a maximum
of finitely many almost-surely finite random variables, $M_N(\omega)<\infty$
almost surely.  Hence $\mathbb P(M_N(\omega)<h)\uparrow1$ as $h\to\infty$, so
there exists a \emph{deterministic} $H$ for which
\[
 \mathbb P(E_H)\ge1-\delta,
 \qquad E_H:=\{\omega:M_N(\omega)<H\}.
\]
Use the same seed for the affine run and the run on this fixed $f_H$.  On
$E_H$, the deterministic transcript induction applies seed by seed and gives
\[
 \mathcal T_t^{H,\omega}=\mathcal T_t^{\ell,\omega}
 \quad(0\le t\le N),
\]
together with equality of the two output maps.  Thus the runs are identical
through the horizon and every encountered gradient has magnitude $g>\eps$.
Therefore success can occur only on the complementary event, whose probability
is at most $\delta$.  This proves part~(b).

\emph{Expected hitting time.}
For each $N$, apply part~(b) with $\delta=1/2$ and let $T$ be the first hitting
time of a queried or returned $\eps$-gradient point.  The selected instance
satisfies $\mathbb P(T>N)\ge1/2$.  Since
\[
 T\ge N\,\mathbf 1_{\{T>N\}},
 \qquad
 \mathbb E T\ge N\mathbb P(T>N)\ge\frac N2,
\]
the supremum of $\mathbb E T$ over the fixed-normalization family is at least
$N/2$ for every $N$, and is therefore infinite.  This proves part~(c).
\end{proof}

The obstruction is informational rather than geometric: the hidden length
$H$ can remain beyond every finite local transcript even when the
scale-invariant quantity $LR/\eps$ is fixed.  In particular, part~(b) shows
that randomization does not make the initialization free: no finite randomized
procedure can be guaranteed, with uniformly high probability, to locate a
point whose gradient differs from $\nabla f(x_0)$ at a cost controlled only
by $LR/\eps$.

The Euclidean parameter-free result of Lan--Ouyang--Zhang begins after a
nondegenerate secant point has been selected.  That initialization carries
information absent from the strict counted model.  The interface in
\eqref{eq:M0} is one sufficient way to supply it; uniqueness and minimality
of the interface remain open.

\subsection{The known-parameter rate and the parameter-free theorem}

Let $x_0\in\R^d$ be the initial point and let
$f:\R^d\to\R$ be differentiable and convex, with
$X^*:=\operatorname*{argmin}f\ne\varnothing$.  Assume
\begin{equation}\label{eq:smoothness}
 \norm{\grad f(x)-\grad f(y)}_q\le L\norm{x-y}_p
 \qquad(x,y\in\R^d),
\end{equation}
and write
\begin{equation}\label{eq:condition}
 R:=\operatorname{dist}_p(x_0,X^*),\qquad
 K:=\frac{LR}{\eps},\qquad \Kbar:=\max\{1,K\}.
\end{equation}

The following proposition isolates the local rate that will be used by the
parameter-free construction.  In this statement $L,R,p$, and $\eps$ are
known, one oracle call returns the exact pair $(f(x),\nabla f(x))$, the
prescribed point $x_0$ is the first query and is charged, subsequent queries
may be arbitrary deterministic functions of the preceding pairs, and the
output must be a queried point.

\begin{proposition}[Known-parameter exponent for finite $p>2$]
\label{prop:pgtwo-optimality}
Fix $2<p<\infty$, let $q=p/(p-1)$, and suppose $L,R>0$.
\begin{enumerate}[label=(\roman*)]
\item If $\operatorname{dist}_p(x_0,X^*)\le R$, a deterministic method
returns a queried point $\widehat x$ with
$\|\nabla f(\widehat x)\|_q\le\eps$ using
\begin{equation}\label{eq:pgtwo-known-upper}
 O_p\!\left(1+\left(\frac{LR}{\eps}\right)^{p/(p+2)}\right)
\end{equation}
pair-oracle calls, with no dimension-dependent factor.
\item For every $d\ge2$, every integer $1\le T\le d$, and every
deterministic exact-pair algorithm, there is an admissible objective with
$\operatorname{dist}_p(x_0,X^*)=R$ such that every one of its first $T$
queries satisfies
\begin{equation}\label{eq:pgtwo-known-lower}
 \|\nabla f(x_t)\|_q
 \ge \frac{LR}{512M_{p,d}T^{1+2/p}},
 \qquad
 M_{p,d}\le C\min\{p,\ln d\}.
\end{equation}
Consequently, whenever the selected integer horizon is at most $d$, the
query lower bound is
\begin{equation}\label{eq:pgtwo-known-lower-count}
 \Omega\!\left(
 \left[\frac{LR}{\eps\min\{p,\ln d\}}\right]^{p/(p+2)}
 \right).
\end{equation}
\end{enumerate}
For fixed $p$, if $d$ is large enough that the selected horizon fits and
$p\le3\ln d$, the two sides have the same polynomial exponent $p/(p+2)$ in
$K=LR/\eps$.
\end{proposition}

The lower exponent, smoothing factor, and horizon/dimension restrictions are
the lower benchmark of Diakonikolas--Guzm\'an
\citep[Sec.~4, Thm.~4 and Cor.~2]{DiakonikolasGuzman2024}.  The proof in
\Cref{sec:above} supplies the dimension-free upper construction and gives the
exact value--gradient completion needed for unrestricted queried outputs.

We now return to the post-initialization secant model.  The same finite local
routine will be run at observable trial values $M,D$, rather than at the
unknown $L,R$ used in \Cref{prop:pgtwo-optimality}.

\begin{theorem}[Parameter-free gradient minimization in $\ell_p$]\label{thm:main}
There is an explicit deterministic method whose numerical inputs are only
\[
 p,d,\eps,x_0,z_0,M_0
\]
and which returns an actually queried point $\widehat x$ satisfying
$\norm{\grad f(\widehat x)}_q\le\eps$.  Its post-initialization pair-oracle
count satisfies
\begin{align}
 N_{\rm or}
 &\le C_p\Kbar^{1/2}
      +C_p\log\!\left(e+\frac{L}{M_0}\right),
 &&1<p<2,\label{eq:mainbelow}\\
 N_{\rm or}
 &\le C\Kbar^{1/2}
      +C\log\!\left(e+\frac{L}{M_0}\right),
 &&p=2,\label{eq:maineuclidean}\\
 N_{\rm or}
 &\le C_p\Kbar^{\frac{p}{p+2}}
      +C_p\log\!\left(e+\frac{L}{M_0}\right),
 &&2<p<\infty.\label{eq:mainabove}
\end{align}
Here $C$ is universal and $C_p$ depends only on the fixed exponent $p$;
none of these constants has polynomial dependence on $d$.
\end{theorem}

This theorem resolves Diakonikolas's general-$\ell_p$ parameter-free extension
question for every fixed $1<p<\infty$
\citep[p.~2]{Diakonikolas2025OpenProblems}.  Below two, it combines the prior
known-$L$ AMD/mirror-dual dynamics with a finite guarded trial and the
unknown-$L,R$ controller; for $p>2$,
\Cref{prop:pgtwo-optimality} supplies the sharp local polynomial exponent and
the controller preserves it while removing $L$ and $R$ from the primitive
inputs.

\paragraph{Scope.}
The theorem concerns one fixed $1<p<\infty$, convex differentiable
objectives, queried output, and the initialization \eqref{eq:M0}.  For
$1<p<2$, the upper bound is dimension-uniform and matches the exponent
$1/2$ in $K$ of the deterministic benchmark in
\citet[Cor.~2]{DiakonikolasGuzman2024}; that benchmark retains $\ln d$ and
dimension/horizon qualifications, and its radius parameterization is not an
unqualified same-model lower theorem for unrestricted queried outputs.  For
$p>2$, the matching statement concerns the
deterministic polynomial exponent in $K$; the lower construction retains the conditions $T\le d$ and
$\min\{p,\ln d\}$, and the simplified fixed-$p$ comparison assumes
sufficiently large $d$.  Randomized optimality, exact dimension or logarithmic
dependence, the endpoints $p=1,\infty$, nonconvex objectives, arbitrary
unqueried outputs, and bit complexity remain outside the result.  The strict
information theorem and the parameter-free theorem use the two information
conventions defined above.

\subsection{Proof structure of the parameter-free theorem}
\label{sec:main-proof-architecture}

The proof of \Cref{thm:main} separates two hidden quantities.  A local smoothness model depends on
the unknown curvature scale $L$, while each local success proof is
conditional on a trial distance exceeding the unknown solution distance
$R$.  An independent two-dimensional
guess grid would multiply the cost of the local routine.  Instead, the three
returned outcomes let each failed trial identify the single coordinate that
must change.

The geometry-specific normalization of Diakonikolas--Guzm\'an is
\begin{equation}\label{eq:dg-geometry}
 \psi_p(x)=
 \begin{cases}
 \dfrac{1}{2(p-1)}\norm{x-x_0}_p^2,&1<p\le2,\\[2mm]
 \dfrac1p\norm{x-x_0}_p^p,&p>2,
 \end{cases}
\end{equation}
The factor $1/[2(p-1)]$ below two makes the squared $\ell_p$ prox
one-strongly convex, while $1/p$ above two is the standard power
normalization.  In the local proofs we sometimes pull these constants outside
the prox function; \eqref{eq:dg-geometry} gives the normalization used by
Diakonikolas--Guzm\'an \citep[Eq.~(18), Thm.~3]{DiakonikolasGuzman2024}.
The initial secant calibration and the scale--distance guess-and-check adapt
the parameter-free procedure of Lan--Ouyang--Zhang
\citep{LanOuyangZhang2026}.  Ito and Fukuda provide a related
adaptive-regularization predecessor for unknown-distance gradient-norm
guarantees \citep{ItoFukuda2021}.  The three local outcomes satisfy
\[
 \Success \quad\text{or}\quad
 \Scale\Rightarrow M<L \quad\text{or}\quad
 \Radius\Rightarrow D<R.
\]
The proof proceeds as follows:
\[
\begin{array}{c}
\boxed{\text{secant input }(z_0,M_0)}
\ \Longrightarrow\
\boxed{\text{calibrated pair }(M_{\rm a},D_{\rm a})}\\[1mm]
\Downarrow\\[-1mm]
\boxed{\text{controller with }\Success,\Scale,\Radius}
\ \Longleftarrow\
\boxed{\text{one finite local routine for each }p\text{-regime}}\\[1mm]
\Downarrow\\[-1mm]
\boxed{\text{radius sum inside each scale epoch}}
\ \Longrightarrow\
\boxed{\text{scale-epoch sum and Theorem~\ref{thm:main}}}.
\end{array}
\tag{PA}
\]
This diagram also fixes the logical order: the controller in
\Cref{sec:certification} is proved before the three local routines are plugged
into it, and the two geometric sums are postponed to \Cref{sec:global}.

\begin{table}[ht]
\centering
\small
\caption{Results in Sections~3--7 used in the proof of \Cref{thm:main}.}
\label{tab:proof-modules}
\begin{tabularx}{\textwidth}{@{}p{0.12\textwidth}p{0.27\textwidth}YY@{}}
\toprule
Section & Result & Inputs and outcomes & Role in the theorem \\
\midrule
\ref{sec:certification} & Scale and radius adaptation
& Secant input and guarded local trial $\mapsto$ calibrated values or one of
$\Success,\Scale,\Radius$ & Identifies which guess must increase \\
\ref{sec:below} & Below-two geometry
& Two guarded mirror phases with a queried terminal point and
$O_p(\kappa^{1/2})$ cost & Supplies \eqref{eq:mainbelow} locally \\
\ref{sec:euclidean} & Euclidean geometry
& $p=2$ finite-data trial with $O(\kappa^{1/2})$ cost
& Supplies \eqref{eq:maineuclidean} locally \\
\ref{sec:above} & Above-two geometry
& Two finite mirror phases with consecutive-pair tests and
$O_p(\kappa^{p/(p+2)})$ cost
& Proves \Cref{prop:pgtwo-optimality} and supplies
\eqref{eq:mainabove} locally \\
\ref{sec:global} & Cost summation and completion
& Visited trial pairs $\mapsto$ two terminal-dominated sums
& Adds calibration overhead and assembles \Cref{thm:main} \\
\bottomrule
\end{tabularx}
\end{table}

Two implications are used throughout.  A $\Scale$ or $\Radius$ label is returned
only with the displayed implication, so the controller never guesses which
hidden coordinate failed.  The above-two section also proves the compatible
deterministic lower benchmark in \Cref{prop:pgtwo-optimality}.  The three
local bounds are combined in \Cref{sec:global}.

\section{Scale and radius adaptation}
\label{sec:certification}

For $g\ne0$, define the $\ell_p$ norming vector
\begin{equation}\label{eq:normingvector}
 [v(g)]_i:=\frac{\operatorname{sgn}(g_i)|g_i|^{q-1}}
 {\norm g_q^{q-1}}.
\end{equation}
Then
\begin{equation}\label{eq:normingidentities}
 \norm{v(g)}_p=1,\qquad \ip g{v(g)}=\norm g_q.
\end{equation}
Smoothness implies
\begin{equation}\label{eq:descent}
 f(y)\le f(x)+\ip{\grad f(x)}{y-x}+\frac L2\norm{y-x}_p^2.
\end{equation}

We use the observable predicates
\begin{align}
 \mathsf U_M(x,y):\quad&
 f(y)\le f(x)+\ip{\grad f(x)}{y-x}+\frac M2\norm{y-x}_p^2,
 \label{eq:upperguard}\\
 \mathsf G_M(x,y):\quad&
 \norm{\grad f(y)-\grad f(x)}_q\le M\norm{y-x}_p.
 \label{eq:gradientguard}\\
 \mathsf C_M(x,y):\quad&
 D_f(x,y)\ge
 \frac{\norm{\grad f(x)-\grad f(y)}_q^2}{2M},
 \qquad
 D_f(x,y):=f(x)-f(y)-\ip{\grad f(y)}{x-y}.
 \label{eq:cocoercivityguard}
\end{align}
Failure of any predicate used by a regime proves $M<L$.  The upper-model/backtracking logic
follows standard universal and adaptive constructions
\citep{Nesterov2015Universal,LanOuyangZhang2026,ItoFukuda2021}; here each
failed test is also returned to the controller.

For the Banach smooth-convex inequality underlying
\eqref{eq:cocoercivityguard}, see \citet[App.~A, Lem.~12]{CohenSidfordTian2021}; for the exact equivalence used with mirror duality here, see
\citet[Lem.~1(iii)]{KimParkOzdaglarDiakonikolasRyu2024}.  Standard
convex-conjugacy background is given by
\citet[Sec.~3.5]{Zalinescu2002}.  The same orientation of the finite quotient
appears in the A2GD line-search test of \citet[Eq.~(10)]{XuChen2026A2GD}.
Indeed, applying the descent inequality to
$f(\cdot)-\langle\nabla f(y),\cdot\rangle$ and minimizing its quadratic upper
model gives
\begin{equation}\label{eq:truecocoercivity}
 D_f(x,y)\ge
 \frac{\norm{\nabla f(x)-\nabla f(y)}_q^2}{2L}.
\end{equation}
Thus $M\ge L$ implies every $\mathsf C_M(x,y)$, and the contrapositive is
\begin{equation}\label{eq:cocoercivity-scale}
 \neg\mathsf C_M(x,y)\quad\Longrightarrow\quad M<L.
\end{equation}
The controller uses only this implication; the regime-specific finite
two-phase success arguments are proved in \Cref{sec:below,sec:above}.

When $p=2$, we additionally use the ordered finite-data interpolation
predicate
\begin{equation}\label{eq:interpguard}
 \mathsf I_M(i,j):\quad
 f_i-f_j-\ip{g_j}{x_i-x_j}
 -\frac1{2M}\norm{g_i-g_j}_2^2\ge0.
\end{equation}
This is the smooth-convex specialization of the exact interpolation
conditions of Taylor--Hendrickx--Glineur
\citep[Cor.~1]{TaylorHendrickxGlineur2017}.  Whenever $M\ge L$, it follows
directly by applying \eqref{eq:descent} to
$f(\cdot)-\ip{\grad f(x_j)}{\cdot}$ and minimizing the quadratic upper model.

All minimizations below involve only explicit norm powers, affine functions,
and returned oracle data.  They are deterministic internal arithmetic and do
not make uncounted calls to $f$.

\subsection{Initial calibration along a norming ray}

Query $x_0$ and write
\[
 f_0=f(x_0),\qquad g_0=\grad f(x_0),\qquad G=\norm{g_0}_q.
\]
Return $x_0$ if $G\le\eps$; henceforth suppose $G>\eps$.  At a closest
minimizer $x^*$, $\grad f(x^*)=0$, so
\begin{equation}\label{eq:Gbound}
 G\le LR,
\end{equation}
and $K>1$.

\begin{lemma}[Initial scale and radius calibration]\label{lem:anchor}
Starting from $M=M_0$, set
\[
 D=G/M,\qquad y=x_0-Dv(g_0),
\]
query $y$, and test
\begin{equation}\label{eq:anchortest}
 f(y)\le f_0-\frac{GD}{2}.
\end{equation}
Double $M$ after a failed test.  The first accepted values
$M_{\rm a}$ and $D_{\rm a}=G/M_{\rm a}$ satisfy
\begin{equation}\label{eq:anchorconclusions}
 M_{\rm a}<2L,\qquad D_{\rm a}\le2R.
\end{equation}
The loop uses at most
$1+\lceil\log_2(L/M_0)\rceil$ pair-oracle calls.
\end{lemma}

The gradient-directed test follows the Euclidean calibration in
\citep[Lem.~4.1 and Algs.~4.1, 4.3]{LanOuyangZhang2026}.  The norming ray gives
its general-norm form, and the constants in \eqref{eq:anchorconclusions} connect it
to the controller.

The proof uses the choice $D=G/M$ to balance the linear decrease $GD$ predicted by the
gradient with the quadratic model penalty $MD^2/2$.  We first show that this
balance makes the tested inequality automatic once $M\ge L$, which proves
finite acceptance and the bound $M_{\rm a}<2L$.  We then compare the accepted
decrease with the affine lower bound supplied by convexity at a closest
minimizer; the comparison eliminates the unknown value $f^*$ and yields
$D_{\rm a}\le2R$.

\begin{proof}
The norming direction satisfies
$\|v(g_0)\|_p=1$ and $\langle g_0,v(g_0)\rangle=G$.  Hence the point
$y=x_0-Dv(g_0)$ obeys
\[
 \norm{y-x_0}_p=D,
 \qquad
 \ip{g_0}{y-x_0}=-GD.
\]
The natural balance for a model with curvature $M$ is
\begin{equation}\label{eq:anchorbalance}
 GD=MD^2,
 \qquad\text{equivalently}\qquad D=\frac GM.
\end{equation}
Substituting these identities into the quadratic model gives
\[
 f_0+\ip{g_0}{y-x_0}+\frac M2\norm{y-x_0}_p^2
 =f_0-GD+\frac M2D^2=f_0-\frac{GD}{2}.
\]
If $M\ge L$, the true descent inequality \eqref{eq:descent} is bounded above
by this $M$-model, so \eqref{eq:anchortest} must pass.  Since the loop starts
at $M_0\le L$ and doubles only after rejection, the first accepted value is
either $M_0$ or twice a rejected value below $L$.  In both cases
$M_{\rm a}<2L$.  The number of queried ray points is at most
\[
 1+\left\lceil\log_2\frac{L}{M_0}\right\rceil,
\]
because after that many tested dyadic values the current scale is at least
$L$ and must be accepted.

It remains to certify a radius from the same accepted observation.  Choose
$x^*\in X^*$ with $\norm{x^*-x_0}_p=R$.  Convexity at $x_0$, followed by
H\"older's inequality, gives the affine lower bound
\[
 f^*\ge f_0+\ip{g_0}{x^*-x_0}\ge f_0-GR.
\]
On the other hand, acceptance at $(M_{\rm a},D_{\rm a})$ gives
\[
 f(y)\le f_0-\frac{GD_{\rm a}}2.
\]
Since $f(y)\ge f^*$, the two bounds can hold simultaneously only if
\[
 f_0-GR\le f_0-\frac{GD_{\rm a}}2,
 \qquad\text{hence}\qquad D_{\rm a}\le2R,
\]
where division by $G$ is legal because the trivial case $G\le\eps$ was
already returned.  This proves both conclusions in
\eqref{eq:anchorconclusions}.
\end{proof}

Thus one accepted, queried decrease supplies both displayed bounds and
initializes the controller.  The norming ray is the $\ell_p$ realization
of the Euclidean calibration logic of Lan--Ouyang--Zhang.

\subsection{The controller}

For the fixed $p$, let $\cT_p(M,D)$ denote the appropriate guarded local
trial from \Cref{sec:below,sec:euclidean,sec:above}.  The secant calibration
and distance-scale guess-and-check adapt Lan--Ouyang--Zhang
\citep[Sec.~4]{LanOuyangZhang2026}.  Ito and Fukuda give a related
unknown-distance adaptive-regularization construction
\citep[Alg.~3, Thm.~4.2]{ItoFukuda2021}.  The local routines return the three
outcomes below.

\begin{algorithm}[H]
\caption{Scale/radius controller}\label{alg:controller}
\begin{algorithmic}[1]
\Require $p,d,\eps,x_0,z_0,M_0$ and the pair oracle for $f$
\State Query $x_0$; if $\norm{\grad f(x_0)}_q\le\eps$, return $x_0$
\State Run \Cref{lem:anchor} and obtain $M_{\rm a},D_{\rm a}$
\State $M\gets M_{\rm a}$
\While{true}
  \State $D\gets \norm{\grad f(x_0)}_q/M$
  \While{true}
    \State $o\gets\cT_p(M,D)$
    \If{$o=(\Success,\widehat x)$} \State \Return $\widehat x$ \EndIf
    \If{$o=\Radius$} \State $D\gets2D$ \EndIf
    \If{$o=\Scale$} \State $M\gets2M$; \State \textbf{break} \EndIf
  \EndWhile
\EndWhile
\end{algorithmic}
\end{algorithm}

\begin{proposition}[Implications of the three trial outcomes]\label{prop:certification}
Suppose each call $\cT_p(M,D)$ is made with $D\ge G/M$ and has exactly three
possible outcomes:
\begin{enumerate}[label=(\roman*)]
\item $\Success$, with a queried $x$ satisfying
      $\norm{\grad f(x)}_q\le\eps$;
\item $\Scale$, accompanied by an actually failed predicate among
      \eqref{eq:upperguard}, \eqref{eq:gradientguard},
      \eqref{eq:cocoercivityguard} for $p\ne2$, and, for $p=2$,
      \eqref{eq:interpguard} or the terminal descent guard
      \eqref{eq:terminaldescentguard};
\item $\Radius$, possible only after every guard passes and the final queried
      gradient exceeds $\eps$, with the proved implication $D<R$.
\end{enumerate}
Then \Cref{alg:controller} doubles only a certified deficient coordinate.
Every visited scale satisfies $M<2L$, every expensive visited radius satisfies
$D\le2R$, and the trial costs within each scale epoch and across scale epochs
are geometrically dominated by their terminal values.  In particular, the
search introduces no rectangular scale--radius multiplier.
\end{proposition}

The proof has two logical directions.  First, each returned
failure must be impossible when $M\ge L$; its contrapositive then certifies
that only the scale coordinate is deficient.  Second, each fully guarded
nonsuccessful local proof establishes the forward implication
$D\ge R\Rightarrow\Success$; its contrapositive certifies that only the radius
coordinate is deficient.  We then propagate these two implications through
the update rule, paying special attention to the reset $D=G/M$ after a scale
increase.  The detailed geometric costs are deferred to
\Cref{lem:amortization}.

\begin{proof}
\emph{Scale returns.}
The descent inequality \eqref{eq:descent} implies $\mathsf U_M(x,y)$ whenever
$M\ge L$, while \eqref{eq:smoothness} implies $\mathsf G_M(x,y)$.  The
Banach inequality \eqref{eq:truecocoercivity} implies
$\mathsf C_M(x,y)$.  The
smooth-convex interpolation inequality preceding the calibration proves every
$\mathsf I_M(i,j)$ under the same hypothesis.  Finally, in the Euclidean
terminal test, $v=u-g/M$ with $g=\grad f(u)$, so \eqref{eq:descent} gives
\begin{align*}
 f(v)
 &\le f(u)-\frac1M\norm g_2^2
       +\frac{L}{2M^2}\norm g_2^2\\
 &\le f(u)-\frac1{2M}\norm g_2^2
 \qquad(M\ge L).
\end{align*}
Thus every possible predicate named in outcome~(ii) holds whenever $M\ge L$.
The forward implication is
\[
 M\ge L\quad\Longrightarrow\quad\text{no }\Scale\text{ outcome}.
\]
Its contrapositive is precisely
\begin{equation}\label{eq:scalecertificate}
 \Scale\quad\Longrightarrow\quad M<L.
\end{equation}
Accordingly the controller doubles the scale only after proving it deficient,
and the new value $2M$ is strictly below $2L$.

\emph{Radius returns.}
For each regime-specific trial, the local proof establishes the forward
implication
\begin{equation}\label{eq:radiusforward}
 D\ge R\ \text{ and every guard passes}
 \quad\Longrightarrow\quad
 \text{the final queried gradient is at most }\eps.
\end{equation}
Outcomes are exhaustive.  A completed guarded trial whose final queried
gradient is larger than $\eps$ therefore satisfies the explicit
contrapositive
\begin{equation}\label{eq:radiuscertificate}
 \Radius\quad\Longrightarrow\quad D<R.
\end{equation}
This is why the controller doubles the radius only after a $\Radius$ outcome.

\emph{Propagation along the visited pairs.}
At the accepted calibration value, the first radius is
$G/M_{\rm a}=D_{\rm a}\le2R$.  If the scale is later doubled from $M$ to
$2M$, the reset is essential:
\[
 D_{\rm new}=\frac{G}{2M}=\frac12\frac GM.
\]
It restores the invariant $D\ge G/M$ at equality and prevents an obsolete
large radius from being carried into the next scale epoch.  Because scales
only increase, every later epoch starts with a radius no larger than
$D_{\rm a}$.  Inside one epoch, the first radius is therefore at most $2R$;
every later radius is twice a predecessor satisfying $D<R$ by
\eqref{eq:radiuscertificate}, and is again at most $2R$.  Hence
\begin{equation}\label{eq:allradii}
 D\le2R
\end{equation}
for every expensive trial, while every visited scale is below $2L$.

Only one coordinate changes at a time: $\Radius$ changes $D$ with $M$ fixed,
and $\Scale$ changes $M$ and resets $D$.  The realized calls therefore form a
single sequence of visited pairs, not all pairs in a rectangular grid.  Along
the radius part of an epoch, $\kappa=MD/\eps$ doubles; after the scale has been
proved too small, the
reset begins the next epoch at $\kappa=G/\eps$.  The local cost functions are
monotone powers of $\kappa$, so both
sequences are geometrically summable.  \Cref{lem:amortization} carries out
the two sums with the stated exponents.  The nested Euclidean precedent is
\citep[Thm.~4.2]{LanOuyangZhang2026}; restart roles also appear in
\citep{ItoFukuda2021,RenegarGrimmer2022}.
\end{proof}

The local solvers therefore never return an undifferentiated failure: the
returned outcome identifies the coordinate proved to be too small.  The next
table separates the geometry-specific work inside each trial from the
three outcomes used by the controller.  In every row, a failed listed guard proves
$M<L$; a completed guarded trial with gradient larger than $\eps$ certifies
$D<R$.

\begin{table}[H]
\centering
\scriptsize
\setlength{\tabcolsep}{3.5pt}
\renewcommand{\arraystretch}{1.18}
\caption{Three geometry-specific routines and the outcomes used by the
controller.}
\label{tab:three-regimes}
\begin{tabularx}{\textwidth}{@{}P{0.08\textwidth}P{0.17\textwidth}P{0.19\textwidth}P{0.19\textwidth}P{0.14\textwidth}Y@{}}
\toprule
Regime & Geometry and regularizer & Local target and calculation & Failed guards yielding $\Scale$ & $\Radius$ implication & Local cost for $\kappa=MD/\eps$ \\
\midrule
$1<p<2$
& One-strong convexity of $\|\cdot\|_p^2/[2(p-1)]$ and its conjugate mirror map
& Guarded AMD value-gap phase followed by the finite anti-diagonal mirror-dual terminal-gradient phase
& Consecutive-pair $\mathsf C_M$ tests in both phases
& Fully guarded nonsuccess implies $D<R$
& $C_p\kappa^{1/2}$ \\
\addlinespace
$p=2$
& Euclidean quadratic estimate sequence; no persistent regularizer in the two-phase trial
& Phase A produces a function gap; guarded OGM-G converts it to a terminal-gradient certificate
& Upper-model guard, every ordered interpolation guard, and the terminal descent guard
& Fully guarded nonsuccess implies $D<R$
& $C\kappa^{1/2}$ \\
\addlinespace
$p>2$
& $p$-uniform convexity of $h=\|\cdot\|_p^p/p$ and the conjugate mirror map
& Phase I reaches a gap of order $\eps D$; the anti-diagonal Phase II controls the terminal $q$-power gradient
& Consecutive-pair $\mathsf C_M$ tests in both phases
& Fully guarded nonsuccess implies $D<R$
& $C_p\kappa^{p/(p+2)}$ \\
\bottomrule
\end{tabularx}
\end{table}

\section[The local solver for 1<p<2]{The local solver for $1<p<2$}
\label{sec:below}

The CFOM/AMD coefficient architecture, its anti-diagonal mirror dual, and the
known-$L$, fixed-horizon $O_p(LR/N^2)$ terminal-gradient specialization below
two are due to Kim--Park--Ozdaglar--Diakonikolas--Ryu
\citep[Secs.~2.3 and 3.1, Thm.~1, and Cor.~3]{KimParkOzdaglarDiakonikolasRyu2024}.
The explicit finite plateau schedule in \eqref{eq:below-weights} is the local
choice used for the guarded trial below.
The squared-$\ell_p$ geometry also appears in the general-norm construction of
Diakonikolas--Guzm\'an \citep[Eq.~(18)]{DiakonikolasGuzman2024}; their older
regularization-to-residual route carries a multiplicative accuracy logarithm.
Here we retain the prior no-log dynamics but expose every finite inequality
that depends on smoothness, test it at the observable trial scale $M$, and
place the two queried phases inside the common $\Success/\Scale/\Radius$
interface.  This guarded realization and its coupling to unknown $L,R$ are
the below-two contribution of the present theorem.
Throughout this section, set $\sigma=p-1$.

\begin{proposition}[The $1<p<2$ trial]\label{prop:belowtrial}
Suppose the pair at $x_0$ is cached, put
$G=\norm{\grad f(x_0)}_q>\eps$, and let $M,D>0$ satisfy $D\ge G/M$.
There is an explicit deterministic finite routine using only $p,\eps,x_0$,
$M,D$, cached oracle data, and new exact value--gradient pair queries.  It has
exactly the following outcomes:
\begin{enumerate}[label=(\roman*)]
\item $\Success$, with a queried point $\widehat x$ satisfying
      $\norm{\grad f(\widehat x)}_q\le\eps$;
\item $\Scale$, accompanied by an actually failed consecutive-pair predicate
      $\mathsf C_M(X,Y)$, and necessarily $M<L$;
\item $\Radius$, possible only after all predicates pass and the terminal
      queried pair has gradient norm larger than $\eps$, and necessarily
      $D<R$.
\end{enumerate}
Its number of new pair queries satisfies
\begin{equation}\label{eq:below-local-count}
 N_{\rm local}
 \le 4\sqrt{\frac{MD}{\sigma\eps}}+2
 \le \left(\frac4{\sqrt\sigma}+2\right)
       \sqrt{\frac{MD}{\eps}}.
\end{equation}
Thus the cost is $O_p(\sqrt{MD/\eps})$, with no dimension-dependent constant.
\end{proposition}

\subsection[Squared ell-p geometry and the observable guard]
{Squared $\ell_p$ geometry and the observable guard}

The squared $\ell_p$ normalization below is the one used in the
general-norm small-gradient construction of Diakonikolas--Guzm\'an
\citep[Eq.~(18) and Sec.~3]{DiakonikolasGuzman2024}.  We repeat the curvature calculation because
its dimension-free constant is load-bearing for the finite trial.

Define the duality map $\mathcal J_p(0)=0$ and, for $x\ne0$,
\[
 [\mathcal J_p(x)]_i
 =\norm{x}_p^{2-p}\operatorname{sgn}(x_i)|x_i|^{p-1}.
\]
Then
\begin{equation}\label{eq:below-J-identities}
 \norm{\mathcal J_p(x)}_q=\norm{x}_p,
 \qquad \ip{\mathcal J_p(x)}x=\norm{x}_p^2,
 \qquad \mathcal J_p^{-1}=\mathcal J_q.
\end{equation}

\begin{lemma}[Dimension-free quadratic geometry]\label{lem:belowgeometry}
The function
\begin{equation}\label{eq:below-h-pair}
 h(x)=\frac1{2\sigma}\norm{x}_p^2
\end{equation}
is one-strongly convex with respect to $\norm{\cdot}_p$.  Its Fenchel
conjugate and mirror map are
\begin{equation}\label{eq:below-hstar}
 h^*(s)=\frac\sigma2\norm{s}_q^2,
 \qquad \grad h^*(s)=\sigma\mathcal J_q(s).
\end{equation}
Consequently, with
$D_\phi(a,b)=\phi(a)-\phi(b)-\ip{\grad\phi(b)}{a-b}$,
\begin{equation}\label{eq:below-conjugate-orientation}
 D_{h^*}(s,t)
 =D_h\bigl(\grad h^*(t),\grad h^*(s)\bigr)
 \ge\frac12
 \norm{\grad h^*(t)-\grad h^*(s)}_p^2.
\end{equation}
\end{lemma}

\begin{proof}
Put $\psi(x)=\norm{x}_p^2/2$.  At a point with nonzero coordinates,
direct differentiation gives, for every direction $a$,
\begin{align*}
 D^2\psi(x)[a,a]
 &=(2-p)\norm{x}_p^{2-2p}
   \left(\sum_i|x_i|^{p-2}x_i a_i\right)^2\\
 &\quad+\sigma\norm{x}_p^{2-p}
   \sum_i|x_i|^{p-2}a_i^2.
\end{align*}
The first term is nonnegative.  Writing
\[
 |a_i|^p=(|x_i|^{p-2}a_i^2)^{p/2}|x_i|^{p(2-p)/2}
\]
and applying H\"older with exponents $2/p$ and $2/(2-p)$ yields
\[
 \norm{x}_p^{2-p}\sum_i|x_i|^{p-2}a_i^2\ge\norm{a}_p^2.
\]
Thus $D^2\psi(x)[a,a]\ge\sigma\norm{a}_p^2$ wherever the Hessian is
classical.  Along a line segment, the only possible singular factors at a
coordinate crossing have the form $|t-t_0|^{p-2}$; they are integrable because
$p>1$.  Integrating the almost-everywhere Hessian bound twice gives
\[
 \psi(y)\ge\psi(x)+\ip{\mathcal J_p(x)}{y-x}
 +\frac\sigma2\norm{y-x}_p^2.
\]
Division by $\sigma$ proves the strong convexity of $h$.

The conjugate formula follows from the scalar optimization of
$\ip{s}{x}-\norm{x}_p^2/(2\sigma)$ and \eqref{eq:below-J-identities}.  The standard
Fenchel equalities $s=\grad h(x)$ and $x=\grad h^*(s)$ give the reversed
Bregman identity in \eqref{eq:below-conjugate-orientation}; its lower bound is the
one-strong convexity just proved.
\end{proof}

The observable predicate $\mathsf C_M$ was defined in
\eqref{eq:cocoercivityguard}.  The Banach smooth-convex remainder
\eqref{eq:truecocoercivity} was proved there by minimizing the quadratic upper
model of $f(\cdot)-\langle\nabla f(Y),\cdot\rangle$; in particular,
\begin{equation}\label{eq:below-scale-guard}
 \neg\mathsf C_M(X,Y)\quad\Longrightarrow\quad M<L.
\end{equation}
This is the only implication used when a below-two trial returns $\Scale$.

\subsection{The finite primal coefficient method}

This subsection works with an arbitrary differentiable convex normalized
function $F:\R^d\to\R$.  Its Bregman remainder is denoted by $D_F$.
The finite coefficient representation and the accelerated mirror-descent
recurrence adapt the CFOM/AMD construction of
Kim--Park--Ozdaglar--Diakonikolas--Ryu
\citep[Secs.~2.3 and 3.1 and Eqs.~(1)--(3)]{KimParkOzdaglarDiakonikolasRyu2024}.
They are expanded here so that the finite inequalities requiring smoothness
are visible and guardable.
Fix an integer $n\ge1$ and use the plateaued weights
\begin{equation}\label{eq:below-weights}
 u_{-1}=0,\qquad
 u_k=\frac{(k+1)^2}{4}\quad(0\le k<n),
 \qquad u_n=u_{n-1}=\frac{n^2}{4},
 \qquad d_k=u_k-u_{k-1}.
\end{equation}
Thus $d_n=0$, while for $0\le k<n$,
\begin{equation}\label{eq:below-weight-PSD}
 d_k=\frac{2k+1}{4},\qquad
 u_k-d_k^2=\frac{4k+3}{16}>0.
\end{equation}

Starting from
$s_0=v_0=x_0^{\rm P}=0$, let $g_k=\grad F(x_k^{\rm P})$ and set, for
$k=0,\ldots,n-1$,
\begin{align}
 s_{k+1}&=s_k-d_kg_k,\label{eq:below-primal-s}\\
 v_{k+1}&=\grad h^*(s_{k+1}),\label{eq:below-primal-v}\\
 x_{k+1}^{\rm P}
 &=\frac{u_k}{u_{k+1}}x_k^{\rm P}
 +\frac{d_{k+1}}{u_{k+1}}v_{k+1}
 +\frac{d_k}{u_{k+1}}(v_{k+1}-v_k).
 \label{eq:below-primal-x}
\end{align}
The point $x_{k+1}^{\rm P}$ is queried before the next update.

\begin{lemma}[Finite guarded primal estimate]\label{lem:below-primal}
Suppose $F$ has a minimizer $z$, every point
$x_0^{\rm P},\ldots,x_n^{\rm P}$ has been queried, and
\begin{equation}\label{eq:below-normalized-guard-primal}
 D_F(x_k^{\rm P},x_{k+1}^{\rm P})
 \ge\frac12\norm{g_k-g_{k+1}}_q^2
 \qquad(0\le k<n).
\end{equation}
Then
\begin{equation}\label{eq:below-primal-gap}
 F(x_n^{\rm P})-\inf F\le\frac{h(z)}{u_n}.
\end{equation}
\end{lemma}

\begin{proof}
Define
\begin{align*}
 \mathcal U_0&=h(z)-u_0D_F(z,x_0^{\rm P}),\\
 \mathcal U_{k+1}
 &=\mathcal U_k-d_{k+1}D_F(z,x_{k+1}^{\rm P})\\
 &\quad-u_k\left[D_F(x_k^{\rm P},x_{k+1}^{\rm P})
 -\frac12\norm{g_k-g_{k+1}}_q^2\right]\\
 &\quad-\left[D_{h^*}(s_k,s_{k+1})
 -\frac12\norm{v_{k+1}-v_k}_p^2\right].
\end{align*}
Every subtracted term is nonnegative: the first by convexity, the second by
\eqref{eq:below-normalized-guard-primal}, and the third by
\eqref{eq:below-conjugate-orientation}.  Hence
$\mathcal U_n\le\mathcal U_0\le h(z)$.

We expose the finite telescope.  Put $A_k=g_k$, $B_k=v_k$, and
$A_{n+1}=0$.  Equation \eqref{eq:below-primal-x} is equivalent to
\begin{equation}\label{eq:below-weighted-x}
 u_{k+1}x_{k+1}^{\rm P}-u_kx_k^{\rm P}
 =d_{k+1}B_{k+1}+d_k(B_{k+1}-B_k).
\end{equation}
Abel summation gives
\begin{align}
 &\sum_{k=0}^{n-1}\ip{s_k-s_{k+1}}{B_{k+1}}
 -\sum_{k=0}^{n}u_k\ip{A_k-A_{k+1}}{x_k^{\rm P}}\notag\\
 &\hspace{2cm}=
 \sum_{k=0}^{n-1}d_k\ip{A_k-A_{k+1}}{B_{k+1}-B_k}.
 \label{eq:below-Abel}
\end{align}
The endpoint terms close because $u_0=d_0$, $x_0^{\rm P}=B_0=0$, and
$d_n=0$.  Expanding the Bregman terms and using \eqref{eq:below-Abel} yields
\begin{equation}\label{eq:below-U-terminal}
 \mathcal U_n
 =u_n\bigl(F(x_n^{\rm P})-F(z)\bigr)
 +h(z)+h^*(s_n)-\ip{s_n}z+\mathcal R_{\rm P},
\end{equation}
where
\begin{align}
 \mathcal R_{\rm P}
 &=\sum_{k=0}^{n-1}\biggl[
 \frac{u_k}{2}\norm{A_k-A_{k+1}}_q^2
 +\frac12\norm{B_k-B_{k+1}}_p^2\notag\\
 &\hspace{3.5cm}
 +d_k\ip{A_k-A_{k+1}}{B_{k+1}-B_k}\biggr].
 \label{eq:below-RP}
\end{align}
For $X=\norm{A_k-A_{k+1}}_q$ and
$Y=\norm{B_{k+1}-B_k}_p$, H\"older's inequality and
\eqref{eq:below-weight-PSD} give the pointwise bound
\begin{align}
 \frac{u_k}{2}X^2+\frac12Y^2-d_kXY
 &=\frac12\left(\sqrt{u_k}X-\frac{d_k}{\sqrt{u_k}}Y\right)^2
 +\frac12\left(1-\frac{d_k^2}{u_k}\right)Y^2\ge0.
 \label{eq:below-mixed-PSD}
\end{align}
Thus $\mathcal R_{\rm P}\ge0$.  The Fenchel term
$h(z)+h^*(s_n)-\ip{s_n}z$ is also nonnegative.  Combining
\eqref{eq:below-U-terminal} with $\mathcal U_n\le h(z)$ proves
\eqref{eq:below-primal-gap}.
\end{proof}

\subsection{The finite anti-diagonal coefficient identity}

For the primal construction define
\begin{align}
 \alpha_{k+1,i}&=d_k\mathbf 1_{\{i=k\}},\label{eq:below-alpha}\\
 c_{0,0}&=1,\notag\\
 c_{k+1,i}
 &=\frac{u_k}{u_{k+1}}c_{k,i}
 +\frac{d_{k+1}+d_k}{u_{k+1}}\mathbf 1_{\{i=k+1\}}
 -\frac{d_k}{u_{k+1}}\mathbf 1_{\{i=k\}},\label{eq:below-c-array}\\
 b_{k+1,i}&=c_{k,i}-c_{k+1,i},
 \qquad b_{0,0}=-1,\label{eq:below-b-array}
\end{align}
where an out-of-range $c_{k,i}$ is zero.  Induction gives
\begin{equation}\label{eq:below-row-sums}
 \sum_{i=0}^kc_{k,i}=1,
 \qquad \sum_{i=0}^{k+1}b_{k+1,i}=0.
\end{equation}
Moreover,
\[
 x_{k+1}^{\rm P}=x_k^{\rm P}-\sum_{i=0}^{k+1}b_{k+1,i}v_i.
\]

We record the exact finite identity needed by the terminal-gradient phase.
It is stated for a general even increment so that the algebraic and geometric
parts remain visibly separate.  The coefficient bijection and its bilinear
telescope are the pointwise proof underlying the published mirror-duality
theorem \citep[Thm.~1 and Eqs.~(8)--(10)]{KimParkOzdaglarDiakonikolasRyu2024}.
The proof below specializes its quadratic increment and retains the full
calculation for guard verification.

\begin{lemma}[Pointwise anti-diagonal residual identity]\label{lem:below-identity}
Let $A_0,\ldots,A_n$ and $B_0,\ldots,B_n$ be free vectors and set
$A_{n+1}=0$.  Define
\[
 X_0=B_0,\qquad
 X_{k+1}=X_k-\sum_{i=0}^{k+1}b_{k+1,i}B_i.
\]
For an even function $\Omega$, put
\begin{align}
 \mathcal R_{\rm P}^{\Omega}(A,B)
 &=\sum_{k=0}^{n-1}\frac{u_k}{2}\norm{A_k-A_{k+1}}_q^2
 +\sum_{k=0}^{n-1}\Omega(B_k-B_{k+1})\notag\\
 &\quad+\sum_{k=0}^{n-1}
 \ip{\sum_{i=0}^k\alpha_{k+1,i}A_i}{B_{k+1}}
 -\sum_{k=0}^{n}u_k\ip{A_k-A_{k+1}}{X_k}.
 \label{eq:below-free-RP}
\end{align}
Let $w_i=1/u_{n-i}$ and, for free vectors $C,D$, define
\begin{align*}
 P_k(C)&=w_{k+1}C_{k+1}
 -\sum_{j=0}^k(w_{j+1}-w_j)C_j,\\
 Q_k(D)&=\sum_{i=0}^k\alpha_{n-i,n-1-k}D_i,
\end{align*}
and
\begin{align}
 \mathcal R_{\rm D}^{\Omega}(C,D)
 &=\sum_{k=0}^{n-1}\frac{w_{k+1}}2\norm{C_k-C_{k+1}}_q^2
 +\sum_{k=0}^{n-1}\Omega(D_k-D_{k+1})\notag\\
 &\quad+\sum_{k=0}^{n}
 \ip{\sum_{i=0}^kb_{n-i,n-k}C_i}{D_k}
 +\sum_{k=0}^{n-1}\ip{P_k(C)}{Q_k(D)}.
 \label{eq:below-free-RD}
\end{align}
The linear change of variables
\begin{align}
 C_0&=u_nA_n,\label{eq:below-map-C0}\\
 C_{n-i}-C_{n-i-1}&=u_i(A_i-A_{i+1})
 \quad(0\le i<n),\label{eq:below-map-Cdiff}\\
 D_i&=B_{n-i}\quad(0\le i\le n)\label{eq:below-map-D}
\end{align}
is a bijection and satisfies
\begin{equation}\label{eq:below-pointwise-identity}
 \mathcal R_{\rm P}^{\Omega}(A,B)
 =\mathcal R_{\rm D}^{\Omega}(C,D).
\end{equation}
\end{lemma}

\begin{proof}
The inverse first recovers $A_n=C_0/u_n$, then recovers the preceding $A_i$
from \eqref{eq:below-map-Cdiff}, and finally reverses $D_i=B_{n-i}$.  If
$i=n-k-1$, then
\[
 \frac{w_{k+1}}2\norm{C_k-C_{k+1}}_q^2
 =\frac{u_i}{2}\norm{A_i-A_{i+1}}_q^2.
\]
Evenness of $\Omega$ gives
\[
 \sum_{k=0}^{n-1}\Omega(D_k-D_{k+1})
 =\sum_{i=0}^{n-1}\Omega(B_i-B_{i+1}).
\]

For the first bilinear block, set $C_{-1}=0$.  Summation by parts in the
recurrence for $X_k$ gives
\begin{equation}\label{eq:below-b-bilinear}
 -\sum_{k=0}^{n}u_k\ip{A_k-A_{k+1}}{X_k}
 =\sum_{k=0}^{n}\sum_{i=0}^kb_{k,i}\ip{C_{n-k}}{B_i}.
\end{equation}
The change of indices $(\ell,j)=(n-i,n-k)$ turns the third term of
\eqref{eq:below-free-RD} into the right-hand side of \eqref{eq:below-b-bilinear}.
For the remaining block, telescoping gives
\begin{equation}\label{eq:below-P-telescope}
 P_k(C)
 =\frac{C_{k+1}-C_k}{u_{n-k-1}}
  +\frac{C_k-C_{k-1}}{u_{n-k}}+\cdots+\frac{C_0}{u_n}
 =A_{n-k-1}.
\end{equation}
The change $(\ell,j)=(n-i-1,n-k-1)$ converts the last term of
\eqref{eq:below-free-RD} into
\[
 \sum_{\ell=0}^{n-1}
 \ip{\sum_{j=0}^{\ell}\alpha_{\ell+1,j}A_j}{B_{\ell+1}},
\]
which is the remaining term of \eqref{eq:below-free-RP}.  The four transformed
blocks prove \eqref{eq:below-pointwise-identity}.
\end{proof}

Take $\Omega(y)=\norm y_p^2/2$.  Equations \eqref{eq:below-free-RP} and
\eqref{eq:below-RP} agree after the primal recurrences are substituted.  Explicitly,
the two bilinear blocks reduce by the same free-vector Abel identity to
\begin{align*}
 &\sum_{k=0}^{n-1}d_k\ip{A_k}{B_{k+1}}
 -\sum_{k=0}^{n}u_k\ip{A_k-A_{k+1}}{X_k}\\
 &\hspace{2cm}=
 \sum_{k=0}^{n-1}d_k
 \ip{A_k-A_{k+1}}{B_{k+1}-B_k}.
\end{align*}
Here $X_0=B_0$, $u_0=d_0$, and $d_n=0$ close the endpoints; no assumption
$B_0=0$ is used in this free identity.  More importantly, the
completed-square calculation \eqref{eq:below-mixed-PSD} applies to
every free $(A,B)$, not merely to an oracle trajectory.  Since the map in
Lemma~\ref{lem:below-identity} is onto, we obtain
\begin{equation}\label{eq:below-RD-nonnegative}
 \mathcal R_{\rm D}^{\Omega}(C,D)\ge0
 \qquad\text{for every free }(C,D).
\end{equation}

\subsection{The finite dual coefficient method}

The anti-diagonal dual recurrence and terminal row relation in this section
are the finite mirror-dual CFOM of
Kim--Park--Ozdaglar--Diakonikolas--Ryu
\citep[Prop.~1 and Cor.~1]{KimParkOzdaglarDiakonikolasRyu2024}.  The query-before-row-update chronology is
spelled out because it is part of the oracle claim.

Starting from a queried point $q_0^{\rm D}$ for a differentiable convex
normalized function $F$, write $G_i=\grad F(q_i^{\rm D})$ and set
\begin{equation}\label{eq:below-dual-init}
 r_0=-b_{n,n}G_0=\frac{d_{n-1}}{u_n}G_0.
\end{equation}
For $k=0,\ldots,n-1$, perform the operations in the displayed order:
\begin{align}
 q_{k+1}^{\rm D}
 &=q_k^{\rm D}-\sum_{i=0}^{k}
   \alpha_{n-i,n-1-k}\grad h^*(r_i)
 =q_k^{\rm D}-d_{n-1-k}\grad h^*(r_k),
 \label{eq:below-dual-q}\\
 &\text{query }q_{k+1}^{\rm D}
   \text{ and set }G_{k+1}=\grad F(q_{k+1}^{\rm D}),
 \label{eq:below-dual-query}\\
 r_{k+1}
 &=r_k-\sum_{i=0}^{k+1}b_{n-i,n-1-k}G_i.
 \label{eq:below-dual-r}
\end{align}
Every gradient in \eqref{eq:below-dual-r} is therefore cached.

\begin{lemma}[Queried terminal-gradient estimate]\label{lem:below-dual}
Suppose every $q_0^{\rm D},\ldots,q_n^{\rm D}$ is queried and
\begin{equation}\label{eq:below-normalized-guard-dual}
 D_F(q_k^{\rm D},q_{k+1}^{\rm D})
 \ge\frac12\norm{G_k-G_{k+1}}_q^2
 \qquad(0\le k<n).
\end{equation}
Then
\begin{equation}\label{eq:below-dual-bound}
 h^*\bigl(\grad F(q_n^{\rm D})\bigr)
 \le\frac{F(q_0^{\rm D})-\inf F}{u_n}.
\end{equation}
The point whose gradient appears here was queried before the terminal identity
was formed.
\end{lemma}

\begin{proof}
First expand \eqref{eq:below-dual-r}.  The coefficient of $G_0$ in $r_n$ is the
negative row sum of $b_n$, the coefficient of $G_j$ for $1\le j<n$ is the
negative row sum of $b_{n-j}$, and these vanish by \eqref{eq:below-row-sums}.  The
coefficient of $G_n$ is $-b_{0,0}=1$.  Hence
\begin{equation}\label{eq:below-terminal-row}
 r_n=G_n=\grad F(q_n^{\rm D}).
\end{equation}
Chronology \eqref{eq:below-dual-query} shows that $q_n^{\rm D}$ is already queried.

Let $w_k=1/u_{n-k}$, $\Delta w_k=w_{k+1}-w_k$, and define
\begin{align*}
 \mathcal V_0&=w_0\bigl(F(q_0^{\rm D})-F(q_n^{\rm D})\bigr),\\
 \mathcal V_{k+1}
 &=\mathcal V_k-\Delta w_kD_F(q_n^{\rm D},q_k^{\rm D})\\
 &\quad-w_{k+1}\left[D_F(q_k^{\rm D},q_{k+1}^{\rm D})
 -\frac12\norm{G_k-G_{k+1}}_q^2\right]\\
 &\quad-\left[D_{h^*}(r_k,r_{k+1})
 -\frac12\norm{\grad h^*(r_{k+1})-\grad h^*(r_k)}_p^2\right].
\end{align*}
The weights $w_k$ are nondecreasing.  Thus every subtracted term is
nonnegative by convexity, \eqref{eq:below-normalized-guard-dual}, and
\eqref{eq:below-conjugate-orientation}; consequently
$\mathcal V_n\le\mathcal V_0$.

For completeness, expand the terminal energy.  With
$Z_i=\grad h^*(r_i)$, equation \eqref{eq:below-dual-q} says
$q_{k+1}^{\rm D}-q_k^{\rm D}=-Q_k(Z)$.  Pure function-value coefficients
cancel because
\[
 w_0+\Delta w_0-w_1=0,\qquad
 \Delta w_k+w_k-w_{k+1}=0,
\]
and the terminal coefficient telescopes.  Reversing the order in
$q_n^{\rm D}-q_k^{\rm D}=-\sum_{\ell=k}^{n-1}Q_\ell(Z)$ gives the remaining
function block
\begin{align}\label{eq:below-dual-function-block}
 &w_0(F(q_0^{\rm D})-F(q_n^{\rm D}))
 -\sum_{k=0}^{n-1}\Delta w_kD_F(q_n^{\rm D},q_k^{\rm D})
 -\sum_{k=0}^{n-1}w_{k+1}D_F(q_k^{\rm D},q_{k+1}^{\rm D})\notag\\
 &\qquad=\sum_{k=0}^{n-1}\ip{P_k(G)}{Q_k(Z)}.
\end{align}
Set $r_{-1}=0$.  The initialization and update give
\[
 \sum_{i=0}^kb_{n-i,n-k}G_i=r_{k-1}-r_k
 \qquad(0\le k\le n).
\]
Expanding the mirror Bregman terms therefore yields
\begin{align}\label{eq:below-dual-mirror-block}
 -\sum_{k=0}^{n-1}D_{h^*}(r_k,r_{k+1})
 &=h^*(r_n)+D_{h^*}(0,r_0)\notag\\
 &\quad+\sum_{k=0}^{n}
 \ip{\sum_{i=0}^kb_{n-i,n-k}G_i}{Z_k}.
\end{align}
Together, \eqref{eq:below-dual-function-block},
\eqref{eq:below-dual-mirror-block}, and the two quadratic increment blocks give
\begin{equation}\label{eq:below-V-terminal}
 \mathcal V_n=h^*(r_n)+D_{h^*}(0,r_0)
 +\mathcal R_{\rm D}^{\Omega}(G,Z),
 \qquad \Omega(y)=\frac12\norm y_p^2.
\end{equation}
The last two terms are nonnegative by the definition of a Bregman divergence
and \eqref{eq:below-RD-nonnegative}.  Hence, using \eqref{eq:below-terminal-row},
\[
 h^*(\grad F(q_n^{\rm D}))
 \le\mathcal V_n\le\mathcal V_0
 \le\frac{F(q_0^{\rm D})-\inf F}{u_n},
\]
which proves \eqref{eq:below-dual-bound}.
\end{proof}

\subsection[The observable trial at (M,D)]{The observable trial at $(M,D)$}

We now give the routine in Proposition~\ref{prop:belowtrial}.  Define
\begin{equation}\label{eq:below-trial-normalization}
 \kappa=\frac{MD}{\eps}>1,
 \qquad \delta=\frac\eps{MD}=\kappa^{-1},
 \qquad
 n=\left\lceil2\sqrt{\frac{MD}{\sigma\eps}}\right\rceil.
\end{equation}

\paragraph{Phase I.}
Use the normalized function
\begin{equation}\label{eq:below-F1}
 F_1(y)=\frac{f(x_0+Dy)-f(x_0)}{MD^2}.
\end{equation}
Set $P_0=x_0$.  The cached physical pair at $P_0$ supplies the pair at $0$.
Run
\eqref{eq:below-primal-s}--\eqref{eq:below-primal-x}.  After constructing
$x_{k+1}^{\rm P}$, query the physical point
\[
 P_{k+1}=x_0+Dx_{k+1}^{\rm P}.
\]
If its returned gradient has norm at most $\eps$, return
$(\Success,P_{k+1})$.  Otherwise test
$\mathsf C_M(P_k,P_{k+1})$; on failure, return $\Scale$ with the two cached
pairs and the failed numerical inequality.  If all $n$ guards pass, put
$Q_0=P_n$; its pair is cached.

\paragraph{Phase II.}
Recenter without a query:
\begin{equation}\label{eq:below-F2}
 F_2(y)=\frac{f(Q_0+Dy)-f(Q_0)}{MD^2}.
\end{equation}
Set $q_0^{\rm D}=0$ and use the cached pair at $Q_0$.  Precompute
\eqref{eq:below-alpha}--\eqref{eq:below-b-array}, initialize \eqref{eq:below-dual-init}, and
run \eqref{eq:below-dual-q}--\eqref{eq:below-dual-r}.  After forming
$q_{k+1}^{\rm D}$, first query the physical point
\[
 Q_{k+1}=Q_0+Dq_{k+1}^{\rm D}.
\]
If its returned gradient has norm at most $\eps$, return
$(\Success,Q_{k+1})$.  Otherwise test
$\mathsf C_M(Q_k,Q_{k+1})$; on failure, return $\Scale$ with the failed
inequality.  On a passing guard, form $r_{k+1}$ from the now cached gradients.
If all $n$ guards pass and no success was returned, return $\Radius$ together
with the cached terminal pair at $Q_n$.

The only nonlinear internal map is the explicit $\sigma\mathcal J_q$ in
\eqref{eq:below-hstar}.  Coefficients may be computed before the oracle interaction.
There is no line minimization, inverse-gradient call, or optimization oracle.

\begin{lemma}[Exact scaling of the guard]\label{lem:below-guard-scaling}
For normalized points $x,y$ and physical points
$X=c+Dx$, $Y=c+Dy$, where
\[
 F(y)=\frac{f(c+Dy)-f(c)}{MD^2},
\]
one has
\begin{equation}\label{eq:below-guard-equivalence}
 D_F(x,y)\ge\frac12
 \norm{\grad F(x)-\grad F(y)}_q^2
 \quad\Longleftrightarrow\quad
 \mathsf C_M(X,Y).
\end{equation}
\end{lemma}

\begin{proof}
Direct differentiation gives
\[
 \grad F(y)=\frac{\grad f(c+Dy)}{MD},
 \qquad
 D_F(x,y)=\frac{D_f(X,Y)}{MD^2}.
\]
Substitution proves the equivalence.
\end{proof}

\begin{proof}[Proof of Proposition~\ref{prop:belowtrial}]
A $\Success$ output is tested from the exact pair just returned, so its
gradient bound and queried-output status are immediate.  A $\Scale$ output
contains an actual failed $\mathsf C_M$ and therefore implies $M<L$ by
\eqref{eq:below-scale-guard}.

It remains to prove the conditional success statement.  Assume $D\ge R$ and
that every guard in both phases passes.  Choose a closest minimizer $x^*$ and
put $z=(x^*-x_0)/D$.  Then $z$ minimizes $F_1$ and
\begin{equation}\label{eq:below-radius-entry}
 \norm z_p\le1,
 \qquad h(z)\le\frac1{2\sigma}.
\end{equation}
By Lemma~\ref{lem:below-guard-scaling}, the Phase-I guards are exactly
\eqref{eq:below-normalized-guard-primal}.  Lemma~\ref{lem:below-primal} gives
\begin{equation}\label{eq:below-phase-one-gap}
 F_1(x_n^{\rm P})-\inf F_1
 \le\frac1{2\sigma u_n}.
\end{equation}
Since Phase II merely recenters and uses the same physical scale,
\[
 F_2(0)-\inf F_2
 =F_1(x_n^{\rm P})-\inf F_1.
\]
The Phase-II guards are exactly \eqref{eq:below-normalized-guard-dual}, so
Lemma~\ref{lem:below-dual}, \eqref{eq:below-hstar}, and
\eqref{eq:below-phase-one-gap} imply
\[
 \frac\sigma2\norm{\grad F_2(q_n^{\rm D})}_q^2
 \le\frac1{2\sigma u_n^2}.
\]
Therefore
\begin{equation}\label{eq:below-terminal-normalized-gradient}
 \norm{\grad F_2(q_n^{\rm D})}_q
 \le\frac1{\sigma u_n}
 =\frac4{\sigma n^2}
 \le\frac\eps{MD}=\delta,
\end{equation}
where the last inequality is the choice of $n$ in
\eqref{eq:below-trial-normalization}.  Returning to physical units gives
\[
 \norm{\grad f(Q_n)}_q
 =MD\norm{\grad F_2(q_n^{\rm D})}_q\le\eps.
\]
The point $Q_n$ was queried before its guard and before $r_n$ was formed.
Thus
\[
 D\ge R\ \text{and all guards pass}\quad\Longrightarrow\quad\Success.
\]
The routine's outcomes are exhaustive.  An all-passing execution whose
queried gradients all exceed $\eps$ can therefore occur only when $D<R$;
this is precisely the implication attached to $\Radius$.

Phase I makes at most $n$ new queries $P_1,\ldots,P_n$.  Phase II reuses
$P_n=Q_0$ and makes at most $n$ new queries $Q_1,\ldots,Q_n$.  Guards use
cached values and gradients.  Hence
\[
 N_{\rm local}\le2n
 \le4\sqrt{\frac{MD}{\sigma\eps}}+2.
\]
Since $MD/\eps>1$, the additive two is at most
$2\sqrt{MD/\eps}$, proving \eqref{eq:below-local-count}.  No norm equivalence or
dimension-dependent constant has been used.
\end{proof}

\section[The Euclidean p=2 branch]{The Euclidean $p=2$ branch}
\label{sec:euclidean}

The Euclidean branch combines an accelerated function-gap phase with OGM-G,
as in Nesterov--Gasnikov--Guminov--Dvurechensky
\citep[Rem.~1, p.~9 and Eq.~(13)]{NesterovGasnikovGuminovDvurechensky2021},
building on Nesterov's regularization route \citep{Nesterov2012}.
The backward recurrence and terminal-gradient certificate are from
Kim--Fessler
\citep[Lem.~6.1, Thm.~6.1, Sec.~6.3]{KimFessler2021}; the all-pairs predicate
used to guard that certificate is \eqref{eq:interpguard}, due to
Taylor--Hendrickx--Glineur.  Here these ingredients are embedded in a finite
observable trial with a terminal guard and the same three outcomes.

The branch has two phases because the two available certificates control
different quantities.  Phase~A uses the trial radius $D$ to turn accelerated
function-value decrease into a bound on $f(U)-f^*$.  Phase~B starts from that
queried point $U$ and applies the OGM-G recurrence, whose finite-data
certificate turns the gap into a bound on the terminal gradient.  Neither
phase knows $R$ or $f^*$: those quantities appear only in the conditional
proof that a valid radius forces success.

The guard design makes the OGM-G certificate checkable under a trial
scale.  Phase~A checks every upper model that its estimate-sequence proof uses.
Phase~B checks all ordered smooth-convex interpolation inequalities on its
finite data and also queries the final gradient step $v_n$ to check terminal
descent.  A failed check therefore proves that the scale guess is too small,
while a completed nonsuccessful two-phase trial proves the same for the
radius guess.

\subsection{Phase A: guarded function-gap estimate sequence}

Set
\[
 A_0=0,\qquad x_0^a=w_0^a=x_0,\qquad
 \Psi_0(x)=\frac M2\norm{x-x_0}_2^2.
\]
At iteration $k$, choose the positive $a_{k+1}$ satisfying
\begin{equation}\label{eq:euclideanweights}
 a_{k+1}^2=A_k+a_{k+1}=:A_{k+1},
\end{equation}
and set
\[
 y_k=\frac{A_kx_k^a+a_{k+1}w_k^a}{A_{k+1}}.
\]
Query $y_k$, write $g_k=\grad f(y_k)$, and define
\begin{align*}
 \Psi_{k+1}(x)
 &=\Psi_k(x)+a_{k+1}[f(y_k)+\ip{g_k}{x-y_k}],\\
 w_{k+1}^a&=\operatorname*{argmin}_{x\in\R^d}\Psi_{k+1}(x),\\
 x_{k+1}^a
 &=\frac{A_kx_k^a+a_{k+1}w_{k+1}^a}{A_{k+1}}.
\end{align*}
Query $x_{k+1}^a$ and require $\mathsf U_M(y_k,x_{k+1}^a)$; failure
returns $\Scale$.

\begin{lemma}[Guarded Euclidean gap]\label{lem:euclideangap}
If the first $m$ guards pass and $D\ge R$, then
\begin{equation}\label{eq:euclideangap}
 f(x_m^a)-f^*\le\frac{MD^2}{2A_m}
 \le\frac{2MD^2}{(m+1)^2}.
\end{equation}
\end{lemma}

The proof establishes the estimate-sequence invariant only along the accepted finite
trajectory.  Strong convexity of $\Psi_k$ supplies a negative quadratic
error, the accepted upper model supplies a positive one, and the weight
identity $a_{k+1}^2=A_{k+1}$ cancels them exactly.  Evaluating the resulting
potential at a closest optimizer introduces $D$, and a scalar induction gives
the explicit lower bound on $A_m$.

\begin{proof}
The function $\Psi_k$ is $M$-strongly convex and minimized at $w_k^a$.
Assume $A_kf(x_k^a)\le\Psi_k^*$, which is immediate for $k=0$.  Strong
convexity gives
\[
 \Psi_k(w_{k+1}^a)
 \ge\Psi_k^*+\frac M2\norm{w_{k+1}^a-w_k^a}_2^2.
\]
Expanding $\Psi_{k+1}$ at $w_{k+1}^a$ and using the induction hypothesis
first gives
\begin{align*}
 \Psi_{k+1}^*
 &\ge A_k f(x_k^a)
 +a_{k+1}\bigl[f(y_k)+\ip{g_k}{w_{k+1}^a-y_k}\bigr]
 +\frac M2\norm{w_{k+1}^a-w_k^a}_2^2.
\end{align*}
Convexity at $y_k$ gives
\begin{align*}
 A_k f(x_k^a)
 &+a_{k+1}\bigl[f(y_k)+\ip{g_k}{w_{k+1}^a-y_k}\bigr]\\
 &\ge A_{k+1}\bigl[f(y_k)+\ip{g_k}{x_{k+1}^a-y_k}\bigr],
\end{align*}
because
\[
 x_{k+1}^a-y_k
 =\frac{a_{k+1}}{A_{k+1}}(w_{k+1}^a-w_k^a)
\]
and $x_{k+1}^a=(A_kx_k^a+a_{k+1}w_{k+1}^a)/A_{k+1}$.
The accepted upper model now produces
\begin{align*}
 \Psi_{k+1}^*
 &\ge A_{k+1}
   \bigl[f(y_k)+\ip{g_k}{x_{k+1}^a-y_k}\bigr]
   +\frac M2\norm{w_{k+1}^a-w_k^a}_2^2\\
 &\ge A_{k+1}f(x_{k+1}^a)
   -\frac{MA_{k+1}}2\norm{x_{k+1}^a-y_k}_2^2
   +\frac M2\norm{w_{k+1}^a-w_k^a}_2^2\\
 &=A_{k+1}f(x_{k+1}^a)
   -\frac M2\left(\frac{a_{k+1}^2}{A_{k+1}}-1\right)
    \norm{w_{k+1}^a-w_k^a}_2^2.
\end{align*}
The second inequality is the accepted upper-model guard.  The coefficient in
parentheses is zero because \eqref{eq:euclideanweights} says
$a_{k+1}^2=A_{k+1}$.  Thus
$A_{k+1}f(x_{k+1}^a)\le\Psi_{k+1}^*$, closing the induction.

Let $x^*$ be a closest optimizer.  Each affine summand in $\Psi_m$ is a
global lower model of $f$, so evaluation at $x^*$ yields
\[
 \Psi_m^*\le\Psi_m(x^*)
 \le\frac M2\norm{x^*-x_0}_2^2+A_mf^*
 \le\frac{MD^2}{2}+A_mf^*.
\]
Combining the two potential inequalities and dividing by $A_m$ gives the
first bound in \eqref{eq:euclideangap}.

It remains to make $A_m$ explicit.  The positive solution at the first step
is $a_1=A_1=1$.  Put $t_k=\sqrt{A_k}$.  Since
$a_{k+1}=\sqrt{A_{k+1}}=t_{k+1}$ and
$A_{k+1}=A_k+a_{k+1}$,
\[
 t_{k+1}^2-t_{k+1}=t_k^2,
 \qquad
 t_{k+1}=\frac{1+\sqrt{1+4t_k^2}}2\ge t_k+\frac12.
\]
Starting from $t_1=1$, induction gives
$t_m\ge(m+1)/2$, hence $A_m\ge(m+1)^2/4$.  Substitution proves the second
bound in \eqref{eq:euclideangap}.
\end{proof}

Phase~A has therefore produced a fully guarded, queried point $U=x_m^a$ whose
function gap is small whenever the radius guess is valid.  The next phase
uses no unknown optimal value operationally; $f^*$ appears only in the
certificate proving what the finite observations imply.

\subsection{Phase B: finite-data OGM-G}

\begin{lemma}[Guarded finite-data OGM-G]\label{lem:ogmg}
Fix $n\ge1$ and define backwards
\begin{align}
 \theta_n&=1,\nonumber\\
 \theta_i&=\frac{1+\sqrt{1+4\theta_{i+1}^2}}2,
 &&i=n-1,\ldots,1,\label{eq:ogmgtheta}\\
 \theta_0&=\frac{1+\sqrt{1+8\theta_1^2}}2.\nonumber
\end{align}
Starting from $u_0=U$, put $v_{-1}=U$.  For $i=0,\ldots,n-1$, query
$u_i$, set $g_i=\grad f(u_i)$ and $v_i=u_i-g_i/M$, and define
\begin{align}
 u_{i+1}
 &=v_i+
 \frac{(\theta_i-1)(2\theta_{i+1}-1)}
 {\theta_i(2\theta_i-1)}(v_i-v_{i-1})\nonumber\\
 &\quad+
 \frac{2\theta_{i+1}-1}{2\theta_i-1}(v_i-u_i).
 \label{eq:ogmgrecurrence}
\end{align}
Query $u_n$, write $g_n=\grad f(u_n)$, and set $v_n=u_n-g_n/M$.
Suppose $\mathsf I_M(i,j)$ holds for every ordered pair among the returned
triples $(u_i,f(u_i),g_i)_{i=0}^n$, and an additional query at $v_n$ satisfies
\begin{equation}\label{eq:terminaldescentguard}
 f(v_n)\le f(u_n)-\frac1{2M}\norm{g_n}_2^2.
\end{equation}
Then
\begin{equation}\label{eq:ogmgbound}
 \norm{g_n}_2^2
 \le\frac{2M[f(U)-f^*]}{\theta_0^2},
 \qquad \theta_0\ge\frac{n+1}{\sqrt2}.
\end{equation}
\end{lemma}

\paragraph{Proof roadmap.}
Kim--Fessler supply the OGM-G coefficient recurrence and terminal-gradient
certificate; Taylor--Hendrickx--Glineur supply the exact smooth-convex
interpolation inequalities checked here.  We must show that these finite
checks alone imply the certificate under the trial scale $M$.  We rewrite the
interpolation inequalities as nonnegative quantities $I_{ij}$, choose
nonnegative multipliers $\nu_i$, and verify the weighted identity
\eqref{eq:ogmgidentity}.  Its function-value coefficients telescope directly;
its inner products telescope after introducing the auxiliary sequence $p_i$.
The terminal descent query supplies the final nonnegative term that
interpolation among the $u_i$ points does not provide.

\begin{proof}
The recurrence and coefficient sequence are the OGM-G construction of
Kim--Fessler \citep[Sec.~6.3]{KimFessler2021}; we include the full finite
certificate because its observability is load-bearing here.  Put
\[
 \psi_i=f(u_i)-f^*-\frac1{2M}\norm{g_i}_2^2,
 \qquad I_{ij}=\psi_i-\psi_j-\ip{g_j}{v_i-v_j}.
\]
Because $v_i=u_i-g_i/M$, direct expansion gives
\begin{align*}
 I_{ij}
 &=f(u_i)-f(u_j)-\ip{g_j}{u_i-u_j}\\
 &\quad-\frac1{2M}
 \bigl(\norm{g_i}_2^2+\norm{g_j}_2^2-2\ip{g_i}{g_j}\bigr)\\
 &=f(u_i)-f(u_j)-\ip{g_j}{u_i-u_j}
   -\frac1{2M}\norm{g_i-g_j}_2^2.
\end{align*}
Thus the checked interpolation inequalities \eqref{eq:interpguard} are
exactly the assertions $I_{ij}\ge0$.  The terminal guard is logically
separate: since $f^*\le f(v_n)$,
\[
 \psi_n=f(u_n)-f^*-\frac1{2M}\norm{g_n}_2^2
 \ge f(v_n)-f^*\ge0.
\]
Without the query at $v_n$, this final nonnegative remainder would not be
certified by interpolation among the $u_i$ points.

Define
\[
 \nu_0=1,\qquad
 \nu_i=\frac{\theta_0^2}{2\theta_i^2}\quad(1\le i\le n).
\]
The backward recurrence implies
\[
 \theta_i^2-\theta_i=\theta_{i+1}^2\quad(1\le i<n),
 \qquad \theta_0^2-\theta_0=2\theta_1^2
\]
and $\theta_i>\theta_{i+1}$.  Hence $\nu_i$ is nondecreasing.  Put
$\delta_i:=\nu_{i+1}-\nu_i\ge0$.  The certificate to be verified is
\begin{align}
 f(U)-f^*-\frac{\theta_0^2}{2M}\norm{g_n}_2^2
 &=\sum_{i=0}^{n-1}\nu_{i+1}I_{i,i+1}\nonumber\\
 &\quad+\sum_{i=0}^{n-1}(\nu_{i+1}-\nu_i)I_{n,i}
 +\psi_n.\label{eq:ogmgidentity}
\end{align}
We first verify the coefficients of the $\psi_i$ terms.  In the first sum,
their coefficients are $\nu_1$ for $\psi_0$,
$\nu_{i+1}-\nu_i$ for $1\le i<n$, and $-\nu_n$ for $\psi_n$.
The second sum subtracts $\delta_i\psi_i$ and adds
$(\sum_i\delta_i)\psi_n=(\nu_n-\nu_0)\psi_n$; the final term adds one
more $\psi_n$.  Since $\nu_0=1$, all intermediate and terminal
coefficients cancel and the remaining coefficient of $\psi_0$ is one.

It remains to verify the inner-product part.  Let $p_0=0$ and define
\[
 p_{i+1}=\left(1-\frac1{\theta_i}\right)p_i+\frac1{\theta_i}g_i.
\]
This recurrence gives
\begin{equation}\label{eq:ogmgpdifference}
 g_i-p_i=\theta_i(p_{i+1}-p_i).
\end{equation}
We next verify the second recurrence identity for $0\le i\le n$:
\begin{equation}\label{eq:ogmgincrement}
 M(v_i-v_{i-1})=-\theta_i(p_i+p_{i+1}).
\end{equation}
For $i=0$, $v_{-1}=u_0$, $v_0=u_0-g_0/M$, $p_0=0$, and
$p_1=g_0/\theta_0$, so both sides equal $-g_0$.  Suppose the identity holds
at some $i<n$.  Write the two coefficients in \eqref{eq:ogmgrecurrence} as
\[
 \alpha_i=\frac{(\theta_i-1)(2\theta_{i+1}-1)}
 {\theta_i(2\theta_i-1)},
 \qquad
 \beta_i=\frac{2\theta_{i+1}-1}{2\theta_i-1}.
\]
Since $v_i-u_i=-g_i/M$, the OGM-G update and the induction hypothesis give
\begin{align*}
 -M(v_{i+1}-v_i)
 &=\alpha_i\theta_i(p_i+p_{i+1})+\beta_i g_i+g_{i+1}\\
 &=(2\theta_{i+1}-1)p_{i+1}+g_{i+1}\\
 &=\theta_{i+1}(p_{i+1}+p_{i+2}).
\end{align*}
The middle equality uses
$g_i=\theta_i p_{i+1}-(\theta_i-1)p_i$ and the displayed values of
$\alpha_i,\beta_i$; the last is the recurrence defining $p_{i+2}$.
This proves the displayed identity through $i=n$.
For $k=1$, the modified initial relation
$\theta_0^2-\theta_0=2\theta_1^2$ gives
$\delta_0=\nu_1/\theta_0$, while $p_1=g_0/\theta_0$; hence
$\delta_0g_0=\nu_1p_1$.  For $k\ge1$, the ordinary theta relation gives
\[
 \nu_{k+1}\left(1-\frac1{\theta_k}\right)=\nu_k,
 \qquad
 \frac{\nu_{k+1}}{\theta_k}=\delta_k.
\]
The base case above gives
$\sum_{j=0}^{0}\delta_jg_j=\nu_1p_1$.  If the identity holds at $k$, then
\begin{align*}
 \sum_{j=0}^{k}\delta_jg_j
 &=\nu_kp_k+\delta_kg_k\\
 &=\nu_{k+1}\left(1-\frac1{\theta_k}\right)p_k
   +\frac{\nu_{k+1}}{\theta_k}g_k
 =\nu_{k+1}p_{k+1}.
\end{align*}
Thus the recurrence for $p_{k+1}$ and the multiplier relations prove
\begin{equation}\label{eq:ogmgcumulative}
 \sum_{j=0}^{k-1}\delta_jg_j=\nu_kp_k
 \qquad(1\le k\le n).
\end{equation}

Now let $\Delta v_k=v_k-v_{k-1}$.  The inner-product contribution on the
right of \eqref{eq:ogmgidentity} is
\begin{align*}
 Q_{\rm ip}
 &=\sum_{k=1}^n\nu_k\ip{g_k}{\Delta v_k}
   -\sum_{i=0}^{n-1}\delta_i\ip{g_i}{v_n-v_i}\\
 &=\sum_{k=1}^n
   \ip{\nu_k g_k-\sum_{i=0}^{k-1}\delta_i g_i}{\Delta v_k}\\
 &=\sum_{k=1}^n\nu_k\ip{g_k-p_k}{\Delta v_k},
\end{align*}
where the second line expands
$v_n-v_i=\sum_{k=i+1}^n\Delta v_k$ and the last line uses
\eqref{eq:ogmgcumulative}.  Applying \eqref{eq:ogmgpdifference} and
\eqref{eq:ogmgincrement} gives
\begin{align*}
 MQ_{\rm ip}
 &=\sum_{k=1}^n\nu_k\theta_k^2
   (\norm{p_k}_2^2-\norm{p_{k+1}}_2^2)\\
 &=\frac12\norm{g_0}_2^2
   -\frac{\theta_0^2}{2}\norm{g_n}_2^2.
\end{align*}
Indeed, $\nu_k\theta_k^2=\theta_0^2/2$,
$p_1=g_0/\theta_0$, and $p_{n+1}=g_n$ because $\theta_n=1$.  Adding this
inner-product contribution to
$\psi_0=f(U)-f^*-\|g_0\|_2^2/(2M)$ proves
\eqref{eq:ogmgidentity} exactly.

Every multiplier in \eqref{eq:ogmgidentity} is nonnegative, every $I_{ij}$
used there has been checked, and $\psi_n\ge0$.  Dropping the nonnegative
right-hand side and rearranging yields the first inequality in
\eqref{eq:ogmgbound}.

Finally, for $i\ge1$,
\[
 \theta_i=\frac{1+\sqrt{1+4\theta_{i+1}^2}}2
 \ge\theta_{i+1}+\frac12.
\]
Since $\theta_n=1$, this gives $\theta_1\ge(n+1)/2$.  The modified first
relation gives $\theta_0^2\ge2\theta_1^2$, and hence
$\theta_0\ge\sqrt2\theta_1\ge(n+1)/\sqrt2$.
\end{proof}

The certificate uses only the queried values and gradients, the recurrence,
and the terminal descent observation.  In the controller, this matters as
much as the numerical rate: if any required inequality fails, the precise
failure itself proves that the current scale is below $L$.

\begin{proposition}[The Euclidean trial]\label{prop:euclideantrial}
Let
\begin{equation}\label{eq:euclideanhorizon}
 m=n=\left\lceil2\sqrt{\frac{MD}{\eps}}\right\rceil.
\end{equation}
Run Phase A for $m$ iterations, put $U=x_m^a$, and then run
\Cref{lem:ogmg} for $n$ iterations.  Return $\Scale$ at the first failed
upper-model, interpolation, or terminal descent guard; otherwise return
$\Success$ if $\norm{g_n}_2\le\eps$ and $\Radius$ if not.  Every returned
cause has the implications in \Cref{prop:certification}, and the trial uses
at most $C\sqrt{MD/\eps}$ pair-oracle calls.
\end{proposition}

Under a valid radius and passing guards, Phase~A supplies a function-gap
bound at the cached point $U$.  The finite-data OGM-G certificate converts
that bound to a squared terminal-gradient estimate.  We substitute the equal
horizons explicitly and obtain the target.  The forward success implication
then gives the $\Radius$ contrapositive, while each of the three guard families
gives the implication attached to $\Scale$.  We finish with the exact
$2m+n+1$ query count.

\begin{proof}
If $D\ge R$ and no guard fails, \Cref{lem:euclideangap} gives
\[
 f(U)-f^*\le\frac{2MD^2}{(m+1)^2}.
\]
Using this as the input to \Cref{lem:ogmg} and then using
$\theta_0^2\ge(n+1)^2/2$ yields
\begin{align*}
 \norm{g_n}_2^2
 &\le\frac{2M[f(U)-f^*]}{\theta_0^2}\\
 &\le\frac{2M}{\theta_0^2}
       \frac{2MD^2}{(m+1)^2}\\
 &\le\frac{8M^2D^2}{(m+1)^2(n+1)^2}.
\end{align*}
Taking square roots gives
\[
 \norm{g_n}_2
 \le\frac{2\sqrt2\,MD}{(m+1)(n+1)}.
\]
Since $m=n=\lceil2\sqrt{MD/\eps}\rceil$,
\[
 (m+1)(n+1)\ge4\frac{MD}{\eps},
 \qquad
 \norm{g_n}_2\le\frac{\eps}{\sqrt2}<\eps.
\]
Thus we have proved the forward implication
\[
 D\ge R\ \text{ and all guards pass}\quad\Longrightarrow\quad\Success.
\]
Its contrapositive proves that a completed guarded nonsuccess has $D<R$.

For the scale implication, an upper-model or interpolation failure is
impossible when $M\ge L$ by \eqref{eq:descent} and
\eqref{eq:interpguard}.  The calculation in
\Cref{prop:certification} shows the same for the terminal descent guard.
Hence every actual guard failure implies $M<L$.

For the oracle count, Phase~A makes at most two pair queries per iteration,
one at $y_k$ and one at $x_{k+1}^a$, for at most $2m$ calls.  Phase~B starts
from $u_0=U=x_m^a$, whose pair is already cached from Phase~A.  It needs at
most the $n$ new iterate pairs $u_1,\ldots,u_n$ and one additional pair at
$v_n$ for \eqref{eq:terminaldescentguard}.  All $(n+1)^2$ ordered
interpolation checks use cached finite data and make no oracle call.  Thus
\[
 N_{\rm trial}\le2m+n+1
 \le C\sqrt{MD/\eps},
\]
where the final queried gradient is $g_n$ itself.
\end{proof}

\section[The local solver and lower bound for p>2]{The local solver and lower bound for $p>2$}
\label{sec:above}

This section proves \Cref{prop:pgtwo-optimality} and constructs the finite
routine used by \Cref{alg:controller}.  For $p>2$, the conjugate exponent
$q=p/(p-1)$ is below two.  The dimension-free geometry is supplied by
$h(x)=\|x\|_p^p/p$, while the finite-horizon anti-diagonal construction turns
a function-gap method into a terminal-gradient method.  The mirror-duality
theorem of Kim et al. uses a quadratically strongly convex
distance-generating function and therefore does not apply directly to this $h$
\citep[Thm.~1 and Cor.~3]{KimParkOzdaglarDiakonikolasRyu2024}.  The proof
keeps the quadratic gradient-increment block, bounds the $p$-power mirror
block through a mixed residual, and balances its accumulated error against
the terminal weight.  Phase~I reduces the gap only to order $\eps D$;
Phase~II converts that gap into a queried terminal-gradient bound.

The finite CFOM and its anti-diagonal coefficient construction are adapted
from Kim--Park--Ozdaglar--Diakonikolas--Ryu.  The two-stage use belongs to the
broader function-gap-then-gradient lineage of
Nesterov--Gasnikov--Guminov--Dvurechensky and Kim--Fessler
\citep{KimParkOzdaglarDiakonikolasRyu2024,
NesterovGasnikovGuminovDvurechensky2021,KimFessler2021}.  We cite the
coefficient identities where they enter the proof; the mixed
residual estimate, its accumulated error, and the accuracy-dependent weight
balance are proved directly.

\subsection[The p-power geometry and finite primal method]{The $p$-power geometry and finite primal method}

For this section put
\begin{equation}\label{eq:above-mirror-pair}
 h(x):=\frac1p\norm{x}_p^p,
 \qquad
 h^*(s):=\frac1q\norm{s}_q^q,
 \qquad
 a_p:=\frac{2^{2-p}}p.
\end{equation}
This is the standard above-two power geometry in the general-norm
construction of Diakonikolas--Guzm\'an
\citep[Eq.~(18), Sec.~3]{DiakonikolasGuzman2024}.  Its conjugate mirror map
is explicit:
\begin{equation}\label{eq:above-mirror-map}
 [\nabla h^*(s)]_j=\operatorname{sgn}(s_j)|s_j|^{q-1}.
\end{equation}
For $D_\phi(x,y)=\phi(x)-\phi(y)-\langle\nabla\phi(y),x-y\rangle$,
coordinatewise uniform convexity and Fenchel conjugacy give
\begin{align}
 D_h(x,y)&\ge a_p\norm{x-y}_p^p,
 \label{eq:above-uniform-remainder}\\
 D_{h^*}(s,t)
 &=D_h\bigl(\nabla h^*(t),\nabla h^*(s)\bigr).
 \label{eq:above-conjugacy}
\end{align}
The first inequality follows by summing
\[
 \frac{|a|^p-|b|^p}{p}-|b|^{p-2}b(a-b)
 \ge \frac{2^{2-p}}p|a-b|^p
\]
over coordinates.  The reversed arguments in
\eqref{eq:above-conjugacy} fix the orientation used later.  In particular,
if $v_k=\nabla h^*(s_k)$, then
\begin{equation}\label{eq:above-oriented-remainder}
 D_{h^*}(s_k,s_{k+1})
 =D_h(v_{k+1},v_k)
 \ge a_p\norm{v_{k+1}-v_k}_p^p.
\end{equation}

We first work with a convex differentiable normalized function $F$ and a
finite horizon $N\ge1$.  Choose positive nondecreasing weights
\begin{equation}\label{eq:above-general-weights}
 u_{-1}=0<u_0\le\cdots\le u_{N-1},
 \qquad u_N=u_{N-1},
 \qquad d_k:=u_k-u_{k-1}.
\end{equation}
The plateau makes $d_N=0$, which will close the terminal endpoint.  Starting
from $s_0=v_0=x_0^{\rm P}=0$, for $k=0,\ldots,N-1$ query
$x_k^{\rm P}$, set $g_k=\nabla F(x_k^{\rm P})$, and compute
\begin{align}
 s_{k+1}&=s_k-d_kg_k,
 \label{eq:above-primal-s}\\
 v_{k+1}&=\nabla h^*(s_{k+1}),
 \label{eq:above-primal-v}\\
 x_{k+1}^{\rm P}
 &=\frac{u_k}{u_{k+1}}x_k^{\rm P}
 +\frac{d_{k+1}}{u_{k+1}}v_{k+1}
 +\frac{d_k}{u_{k+1}}(v_{k+1}-v_k).
 \label{eq:above-primal-x}
\end{align}
After the final update, query $x_N^{\rm P}$ and set
$g_N=\nabla F(x_N^{\rm P})$.  Thus the endpoint chronology is part of the
algorithm: the final point is queried before it can be returned.

We use the finite CFOM coefficient representation of
\citet[Sec.~3.1, Eqs.~(1)--(3) and Prop.~1]
{KimParkOzdaglarDiakonikolasRyu2024}.  Define
\begin{align}
 \alpha_{k+1,i}&:=d_k\mathbf 1_{\{i=k\}},
 \label{eq:above-alpha}\\
 c_{0,0}&:=1,\nonumber\\
 c_{k+1,i}
 &:=\frac{u_k}{u_{k+1}}c_{k,i}
 +\frac{d_{k+1}+d_k}{u_{k+1}}\mathbf 1_{\{i=k+1\}}
 -\frac{d_k}{u_{k+1}}\mathbf 1_{\{i=k\}},
 \label{eq:above-c-array}\\
 b_{k+1,i}&:=c_{k,i}-c_{k+1,i},
 \qquad b_{0,0}:=-1,
 \label{eq:above-b-array}
\end{align}
where $c_{k,i}=0$ outside $0\le i\le k$.  Then
\begin{align}
 s_{k+1}&=s_k-\sum_{i=0}^k\alpha_{k+1,i}g_i,
 \nonumber\\
 x_k^{\rm P}&=\sum_{i=0}^kc_{k,i}v_i,
 \qquad
 x_{k+1}^{\rm P}=x_k^{\rm P}
 -\sum_{i=0}^{k+1}b_{k+1,i}v_i,
 \label{eq:above-cfom-form}
\end{align}
and induction in \eqref{eq:above-c-array} gives
\begin{equation}\label{eq:above-row-sums}
 \sum_{i=0}^kc_{k,i}=1,
 \qquad
 \sum_{i=0}^{k+1}b_{k+1,i}=0.
\end{equation}

\subsection{The primal residual and its accumulated error}

Let $z$ minimize $F$.  Along a transcript on which the normalized form of
$\mathsf C_M$ holds at every consecutive pair, define
\begin{align}
 \mathcal U_0&:=h(z)-u_0D_F(z,x_0^{\rm P}),
 \label{eq:above-U0}\\
 \mathcal U_{k+1}
 &:=\mathcal U_k-d_{k+1}D_F(z,x_{k+1}^{\rm P})\nonumber\\
 &\quad-u_k\left[D_F(x_k^{\rm P},x_{k+1}^{\rm P})
 -\frac12\norm{g_k-g_{k+1}}_q^2\right]\nonumber\\
 &\quad-\left[D_{h^*}(s_k,s_{k+1})
 -a_p\norm{v_{k+1}-v_k}_p^p\right].
 \label{eq:above-Ustep}
\end{align}
Every subtracted term is nonnegative: the first by convexity, the second by
the accepted consecutive-pair test, and the third by
\eqref{eq:above-oriented-remainder}.  Expanding the Bregman divergences gives
\begin{equation}\label{eq:above-Uterminal}
 \mathcal U_N
 =u_N\bigl(F(x_N^{\rm P})-F(z)\bigr)
 +h(z)+h^*(s_N)-\langle s_N,z\rangle+\mathsf R_P,
\end{equation}
where
\begin{align}
 \mathsf R_P&=\sum_{k=0}^{N-1}\rho_k,
 \label{eq:above-primal-residual}\\
 \rho_k&=\frac{u_k}{2}\norm{g_k-g_{k+1}}_q^2
 +a_p\norm{v_{k+1}-v_k}_p^p\nonumber\\
 &\qquad+d_k\langle g_k-g_{k+1},v_{k+1}-v_k\rangle.
 \label{eq:above-mixed-residual}
\end{align}

The finite expansion fixes both the cross term and its sign.  With
$A_k=g_k$, $B_k=v_k$, and $A_{N+1}=0$, the update gives
\begin{equation}\label{eq:above-weighted-x}
 u_{k+1}x_{k+1}^{\rm P}-u_kx_k^{\rm P}
 =d_{k+1}B_{k+1}+d_k(B_{k+1}-B_k).
\end{equation}
Abel summation therefore yields
\begin{align}
 &\sum_{k=0}^{N-1}\langle s_k-s_{k+1},B_{k+1}\rangle
 -\sum_{k=0}^{N}u_k\langle A_k-A_{k+1},x_k^{\rm P}\rangle\nonumber\\
 &\hspace{1.5cm}=\sum_{k=0}^{N-1}d_k
 \langle A_k-A_{k+1},B_{k+1}-B_k\rangle.
 \label{eq:above-Abel}
\end{align}
Here $u_0=d_0$, $x_0^{\rm P}=B_0$, and $d_N=0$ close the two endpoints;
the same plateau cancels the remaining anchor
$s_N-s_0+\sum_{i=0}^Nd_iA_i$.

The residual in \eqref{eq:above-mixed-residual} can be negative, so its
negative part must be bounded before the two phases are scheduled.  Put
\begin{equation}\label{eq:above-error-constants}
 s_p:=\frac{p}{p-2},
 \qquad
 \kappa_p:=\frac{p-2}{2p}(pa_p)^{-2/(p-2)}.
\end{equation}
For
$X=\|A_k-A_{k+1}\|_q$ and
$Y=\|B_{k+1}-B_k\|_p$, H\"older's inequality gives
\begin{align}
 \rho_k
 &\ge\frac{u_k}{2}X^2+a_pY^p-d_kXY\nonumber\\
 &\ge a_pY^p-\frac{d_k^2}{2u_k}Y^2.
 \label{eq:above-first-minimization}
\end{align}
The first line displays the two stabilizing terms and the cross term produced
by \eqref{eq:above-Abel}.  Minimizing the quadratic expression over $X$
gives the second line.  For $b=d_k^2/(2u_k)$, the remaining scalar problem is
\begin{equation}\label{eq:above-scalar-supremum}
 \sup_{Y\ge0}\{bY^2-a_pY^p\}
 =\frac{p-2}{p}b\left(\frac{2b}{pa_p}\right)^{2/(p-2)}.
\end{equation}
Indeed, its positive stationary point satisfies
$Y^{p-2}=2b/(pa_p)$, and substitution produces the displayed value.  Thus
\begin{equation}\label{eq:above-EN}
 \mathsf R_P\ge-\mathcal E_N,
 \qquad
 \mathcal E_N:=\kappa_p\sum_{k=0}^{N-1}
 \left(\frac{d_k^2}{u_k}\right)^{s_p}.
\end{equation}
The derivation used only free vectors with the fixed coefficient arrays, so
the lower bound is universal.  Since the Fenchel block in
\eqref{eq:above-Uterminal} is nonnegative and
$\mathcal U_N\le\mathcal U_0\le h(z)$, we obtain
\begin{equation}\label{eq:above-primal-gap}
 F(x_N^{\rm P})-\inf F\le\frac{h(z)+\mathcal E_N}{u_N}.
\end{equation}

\subsection{The anti-diagonal residual identity}

The coefficient and sequence bijection in this subsection is the one used in
the proof of mirror duality
\citep[Thm.~1, Eq.~(8), and Eqs.~(4)--(10)]
{KimParkOzdaglarDiakonikolasRyu2024}.  We spell out its four blocks because
the even $p$-power increment and the error \eqref{eq:above-EN} must be
transported pointwise.

For free vectors $A_0,\ldots,A_N$ and $B_0,\ldots,B_N$, set $A_{N+1}=0$ and
\begin{equation}\label{eq:above-free-X}
 X_0=B_0,
 \qquad
 X_{k+1}=X_k-\sum_{i=0}^{k+1}b_{k+1,i}B_i.
\end{equation}
For an even function $\Omega$, define
\begin{align}
 \mathsf R_P^\Omega(A,B)
 &:=\sum_{k=0}^{N-1}\frac{u_k}{2}\norm{A_k-A_{k+1}}_q^2
 +\sum_{k=0}^{N-1}\Omega(B_k-B_{k+1})\nonumber\\
 &\quad+\sum_{k=0}^{N-1}
 \left\langle\sum_{i=0}^k\alpha_{k+1,i}A_i,B_{k+1}\right\rangle
 -\sum_{k=0}^{N}u_k\langle A_k-A_{k+1},X_k\rangle.
 \label{eq:above-free-RP}
\end{align}
Put $w_i=1/u_{N-i}$.  For free $C,D$, let
\begin{align}
 P_k(C)&:=w_{k+1}C_{k+1}
 -\sum_{j=0}^k(w_{j+1}-w_j)C_j,
 \label{eq:above-PC}\\
 Q_k(D)&:=\sum_{i=0}^k\alpha_{N-i,N-1-k}D_i,
 \label{eq:above-QD}
\end{align}
and
\begin{align}
 \mathsf R_D^\Omega(C,D)
 &:=\sum_{k=0}^{N-1}\frac{w_{k+1}}2\norm{C_k-C_{k+1}}_q^2
 +\sum_{k=0}^{N-1}\Omega(D_k-D_{k+1})\nonumber\\
 &\quad+\sum_{k=0}^{N}
 \left\langle\sum_{i=0}^kb_{N-i,N-k}C_i,D_k\right\rangle
 +\sum_{k=0}^{N-1}\langle P_k(C),Q_k(D)\rangle.
 \label{eq:above-free-RD}
\end{align}

\begin{lemma}[Pointwise finite residual identity]
\label{lem:above-pointwise}
For every $A,B$, define $C,D$ by
\begin{align}
 C_0&=u_NA_N,
 \label{eq:above-map-C0}\\
 C_{N-i}-C_{N-i-1}&=u_i(A_i-A_{i+1}),
 \qquad 0\le i<N,
 \label{eq:above-map-Cdiff}\\
 D_i&=B_{N-i},
 \qquad 0\le i\le N.
 \label{eq:above-map-D}
\end{align}
This transformation is a bijection and, for every even $\Omega$,
\begin{equation}\label{eq:above-pointwise-identity}
 \mathsf R_P^\Omega(A,B)=\mathsf R_D^\Omega(C,D).
\end{equation}
\end{lemma}

\begin{proof}
The inverse first recovers $A_N=C_0/u_N$, then recovers the preceding $A_i$
from \eqref{eq:above-map-Cdiff}, and finally sets $B_i=D_{N-i}$.  If
$i=N-k-1$, the quadratic block transforms as
\begin{equation}\label{eq:above-map-quadratic}
 \frac{w_{k+1}}2\norm{C_k-C_{k+1}}_q^2
 =\frac{u_i}{2}\norm{A_i-A_{i+1}}_q^2.
\end{equation}
Because $\Omega$ is even, reversal gives
\begin{equation}\label{eq:above-map-even}
 \sum_{k=0}^{N-1}\Omega(D_k-D_{k+1})
 =\sum_{i=0}^{N-1}\Omega(B_i-B_{i+1}).
\end{equation}

For the first bilinear block, define $C_{-1}=0$.  Summation by parts using
\eqref{eq:above-free-X} gives
\begin{equation}\label{eq:above-map-bilinear-b}
 -\sum_{k=0}^{N}u_k\langle A_k-A_{k+1},X_k\rangle
 =\sum_{k=0}^{N}\sum_{i=0}^kb_{k,i}
 \langle C_{N-k},B_i\rangle.
\end{equation}
The convention $b_{0,0}=-1$ supplies the initial endpoint.  The change of
indices $(\ell,j)=(N-i,N-k)$ turns the third term of
\eqref{eq:above-free-RD} into the right-hand side of
\eqref{eq:above-map-bilinear-b}.  For the remaining block, telescoping in
\eqref{eq:above-PC} gives
\begin{equation}\label{eq:above-map-P}
 P_k(C)=\frac{C_{k+1}-C_k}{u_{N-k-1}}
 +\frac{C_k-C_{k-1}}{u_{N-k}}+\cdots+\frac{C_0}{u_N}
 =A_{N-k-1}.
\end{equation}
The change $(\ell,j)=(N-i-1,N-k-1)$ then converts the last term of
\eqref{eq:above-free-RD} into
\[
 \sum_{\ell=0}^{N-1}
 \left\langle\sum_{j=0}^{\ell}\alpha_{\ell+1,j}A_j,
 B_{\ell+1}\right\rangle,
\]
which is the remaining primal bilinear block.  Equations
\eqref{eq:above-map-quadratic}--\eqref{eq:above-map-P} prove the identity.
\end{proof}

Taking $\Omega(z)=a_p\|z\|_p^p$, the universal primal bound
\eqref{eq:above-EN} and the inverse bijection give, for every free dual
assignment,
\begin{equation}\label{eq:above-dual-residual-lower}
 \mathsf R_D^\Omega(C,D)\ge-\mathcal E_N.
\end{equation}
Thus the transformation preserves the displayed residual formed from the
oriented $p$-power lower remainders.

\subsection{The dual method and its queried endpoint}

The Phase-II coefficient array is the anti-diagonal transform of the Phase-I
array in the sense of
\citet[Secs.~3.1--3.2]{KimParkOzdaglarDiakonikolasRyu2024}.  Starting from a
queried point $q_0^{\rm D}$, write $G_i=\nabla F(q_i^{\rm D})$ and set
\begin{equation}\label{eq:above-dual-init}
 r_0=-b_{N,N}G_0=\frac{d_{N-1}}{u_N}G_0.
\end{equation}
For $k=0,\ldots,N-1$, perform the following operations in order:
\begin{align}
 q_{k+1}^{\rm D}
 &=q_k^{\rm D}-\sum_{i=0}^{k}\alpha_{N-i,N-1-k}\nabla h^*(r_i)\nonumber\\
 &=q_k^{\rm D}-d_{N-1-k}\nabla h^*(r_k),
 \label{eq:above-dual-q}\\
 &\text{query }q_{k+1}^{\rm D}
 \text{ and set }G_{k+1}=\nabla F(q_{k+1}^{\rm D}),
 \label{eq:above-dual-query}\\
 r_{k+1}
 &=r_k-\sum_{i=0}^{k+1}b_{N-i,N-1-k}G_i.
 \label{eq:above-dual-r}
\end{align}
Every gradient in \eqref{eq:above-dual-r} has already been returned.  The
terminal row identity is the CFOM endpoint relation of Kim et al.
\citep[Prop.~1, Lem.~3, and Cor.~1]
{KimParkOzdaglarDiakonikolasRyu2024}.  Expanding
\eqref{eq:above-dual-r}, the coefficient of $G_0$ is the negative row sum of
$b_N$, the coefficient of $G_j$ for $1\le j<N$ is the negative row sum of
$b_{N-j}$, and each vanishes by \eqref{eq:above-row-sums}; the coefficient of
$G_N$ is $-b_{0,0}=1$.  Hence
\begin{equation}\label{eq:above-terminal-gradient}
 r_N=G_N=\nabla F(q_N^{\rm D}).
\end{equation}
The point $q_N^{\rm D}$ is queried before this identity is formed.

For $w_k=1/u_{N-k}$, define
\begin{equation}\label{eq:above-V0}
 \mathcal V_0=w_0\bigl(F(q_0^{\rm D})-F(q_N^{\rm D})\bigr)
\end{equation}
and
\begin{align}
 \mathcal V_{k+1}
 &:=\mathcal V_k-(w_{k+1}-w_k)D_F(q_N^{\rm D},q_k^{\rm D})\nonumber\\
 &\quad-w_{k+1}\left[D_F(q_k^{\rm D},q_{k+1}^{\rm D})
 -\frac12\norm{G_k-G_{k+1}}_q^2\right]\nonumber\\
 &\quad-\left[D_{h^*}(r_k,r_{k+1})
 -a_p\norm{\nabla h^*(r_{k+1})-\nabla h^*(r_k)}_p^p\right].
 \label{eq:above-Vstep}
\end{align}
The first block is nonnegative by convexity, the second by the passed
consecutive-pair test, and the last by
\eqref{eq:above-oriented-remainder}.  We next expose the telescope that
identifies the terminal energy.  Write
$\Delta w_k=w_{k+1}-w_k$ and recall from
\eqref{eq:above-dual-q} that
$q_{k+1}^{\rm D}-q_k^{\rm D}=-Q_k(Z)$.  After the function Bregman terms are
expanded, their pure function-value coefficients cancel:
\[
 \bigl(w_0+\Delta w_0-w_1\bigr)F(q_0^{\rm D})
 +\sum_{k=1}^{N-1}\bigl(\Delta w_k+w_k-w_{k+1}\bigr)F(q_k^{\rm D})
 +\left(-w_0-\sum_{k=0}^{N-1}\Delta w_k+w_N\right)F(q_N^{\rm D})=0.
\]
The remaining inner products can be collected by reversing the order of the
sum over $q_N^{\rm D}-q_k^{\rm D}=-\sum_{\ell=k}^{N-1}Q_\ell(Z)$:
\begin{align*}
 &w_0\bigl(F(q_0^{\rm D})-F(q_N^{\rm D})\bigr)
 -\sum_{k=0}^{N-1}\Delta w_kD_F(q_N^{\rm D},q_k^{\rm D})
 -\sum_{k=0}^{N-1}w_{k+1}D_F(q_k^{\rm D},q_{k+1}^{\rm D})\\
 &\qquad=
 \sum_{k=0}^{N-1}
 \left\langle
 w_{k+1}G_{k+1}-\sum_{j=0}^k\Delta w_jG_j,
 Q_k(Z)\right\rangle
 =\sum_{k=0}^{N-1}\langle P_k(G),Q_k(Z)\rangle.
\end{align*}
This is the second bilinear block in \eqref{eq:above-free-RD}.

For the mirror terms, set $r_{-1}=0$.  The initialization
\eqref{eq:above-dual-init} and update \eqref{eq:above-dual-r} give, for
$0\le k\le N$,
\[
 \sum_{i=0}^kb_{N-i,N-k}G_i=r_{k-1}-r_k.
\]
Using the displayed orientation of $D_{h^*}$ and $Z_k=\nabla h^*(r_k)$
therefore yields
\begin{align*}
 -\sum_{k=0}^{N-1}D_{h^*}(r_k,r_{k+1})
 &=h^*(r_N)-h^*(r_0)
   +\sum_{k=1}^{N}\langle r_{k-1}-r_k,Z_k\rangle\\
 &=h^*(r_N)+D_{h^*}(0,r_0)
   +\sum_{k=0}^{N}
   \left\langle\sum_{i=0}^kb_{N-i,N-k}G_i,Z_k\right\rangle.
\end{align*}
Here the $k=0$ summand is $\langle-r_0,Z_0\rangle$, which combines with
$D_{h^*}(0,r_0)=-h^*(r_0)+\langle Z_0,r_0\rangle$; hence no initial or
terminal mirror term is missing.  The quadratic gradient increments and the
$p$-power increments in \eqref{eq:above-Vstep} are already the first two
blocks of \eqref{eq:above-free-RD}.  Combining the three groups gives
\begin{equation}\label{eq:above-Vterminal}
 \mathcal V_N=h^*(r_N)+D_{h^*}(0,r_0)
 +\mathsf R_D^\Omega(G,Z),
 \qquad Z_i=\nabla h^*(r_i).
\end{equation}
The pointwise identity in \Cref{lem:above-pointwise} is used here only
through the residual lower bound \eqref{eq:above-dual-residual-lower}; no
reversal identity for the full Bregman potential is invoked.  Dropping the
nonnegative second term and using $F(q_N^{\rm D})\ge\inf F$ yields
\begin{equation}\label{eq:above-dual-bound}
 h^*\bigl(\nabla F(q_N^{\rm D})\bigr)
 \le\frac{F(q_0^{\rm D})-\inf F}{u_N}+\mathcal E_N.
\end{equation}

\subsection{Weight balance and the two phases}

For a horizon $N$ and scale $\gamma>0$, choose the plateaued weights
\begin{equation}\label{eq:above-quadratic-weights}
 u_k=\gamma(k+1)^2\quad(0\le k<N),
 \qquad u_N=\gamma N^2.
\end{equation}
Then $d_k^2/u_k\le4\gamma$.  Define
\begin{equation}\label{eq:above-rate-constants}
 B_p:=4^{s_p}\kappa_p,
 \qquad a:=\frac{p-2}{p},
 \qquad r:=\frac{p+2}{p},
 \qquad c_p:=(2B_p)^{-a}.
\end{equation}
Given an error budget $\eta>0$, set
\begin{equation}\label{eq:above-gamma}
 \gamma=\left(\frac{\eta}{2B_pN}\right)^a.
\end{equation}
The error sum and the endpoint weight now have compatible scales:
\begin{align}
 \mathcal E_N
 &\le B_p\gamma^{p/(p-2)}N\le\frac\eta2,
 \label{eq:above-error-budget}\\
 u_N&=\gamma N^2=c_p\eta^{(p-2)/p}N^{(p+2)/p}.
 \label{eq:above-weight-growth}
\end{align}
The power $(p+2)/p$ in \eqref{eq:above-weight-growth} determines the final
complexity exponent $p/(p+2)$.

In Phase~I the goal is to reduce the function gap to the scale needed by the
terminal-gradient phase.  Suppose $F$ has a minimizer $z$ with
$\|z\|_p\le1$.  Use $\eta_F=1/p$.  Since $h(z)\le1/p$,
\eqref{eq:above-primal-gap} and \eqref{eq:above-weight-growth} give
\begin{equation}\label{eq:above-phase-one-gap}
 F(x_N^{\rm P})-\inf F\le H_pN^{-(p+2)/p}
\end{equation}
where
\begin{equation}\label{eq:above-Hp}
 H_p:=\frac{3p^a}{2pc_p}.
\end{equation}
Therefore
\begin{equation}\label{eq:above-NF}
 N_F:=\left\lceil\left(\frac{H_p}{\delta}\right)^{p/(p+2)}\right\rceil
\end{equation}
produces a queried point $y_F=x_{N_F}^{\rm P}$ with
\begin{equation}\label{eq:above-Delta}
 \Delta:=F(y_F)-\inf F\le\delta.
\end{equation}

Phase~II changes the target: the terminal energy now controls
$h^*(\nabla F)=\|\nabla F\|_q^q/q$.  Translate the objective by $y_F$ and
reuse its returned pair as the initial pair.  With
\begin{equation}\label{eq:above-etaD}
 \eta_D:=\frac{\delta^q}{q},
\end{equation}
the dual estimate becomes
\begin{equation}\label{eq:above-phase-two-bound}
 h^*(\nabla F(q_N^{\rm D}))
 \le\frac{\Delta}{c_p\eta_D^aN^{(p+2)/p}}+\frac{\eta_D}{2}.
\end{equation}
The first term is at most $\eta_D/2$ once
\begin{align}
 N^{(p+2)/p}
 &\ge\frac{2\Delta}{c_p\eta_D^{1+a}}
 \le\frac{J_p}{\delta},
 \label{eq:above-phase-two-condition}\\
 q(1+a)&=\frac{p}{p-1}\left(1+\frac{p-2}{p}\right)=2,
 \label{eq:above-exponent-identity}
\end{align}
where
\begin{equation}\label{eq:above-Jp}
 J_p:=\frac{2q^{1+a}}{c_p}.
\end{equation}
Thus choose
\begin{equation}\label{eq:above-ND}
 N_D:=\left\lceil\left(\frac{J_p}{\delta}\right)^{p/(p+2)}\right\rceil.
\end{equation}
Equations \eqref{eq:above-phase-two-bound}--\eqref{eq:above-ND} give
\begin{equation}\label{eq:above-normalized-terminal}
 \frac1q\norm{\nabla F(q_{N_D}^{\rm D})}_q^q
 \le\frac{\delta^q}{q},
 \qquad
 \norm{\nabla F(q_{N_D}^{\rm D})}_q\le\delta.
\end{equation}
Both phases therefore use
$O_p(\delta^{-p/(p+2)})$ updates, but they spend their error budgets on
different quantities: Phase~I on a gap of order $\delta$, and Phase~II on a
$q$-power terminal certificate.

\subsection{The finite trial at guessed curvature and distance}

We now express the two phases using only the trial values $M,D$ and returned
oracle pairs.  Let
$G=\|\nabla f(x_0)\|_q>\eps$, assume $D\ge G/M$, and set
\begin{equation}\label{eq:above-trial-normalization}
 \kappa:=\frac{MD}{\eps}>1,
 \qquad
 \delta:=\frac{\eps}{MD}=\kappa^{-1},
 \qquad
 F_{M,D}(y):=\frac{f(x_0+Dy)-f(x_0)}{MD^2}.
\end{equation}
The exact normalized pair is obtained from the physical pair by
\begin{equation}\label{eq:above-pair-rescaling}
 \left(
 \frac{f(x_0+Dy)-f(x_0)}{MD^2},
 \frac{\nabla f(x_0+Dy)}{MD}
 \right).
\end{equation}
For physical points $X=x_0+Dx$ and $Y=x_0+Dy$, direct substitution gives
\begin{equation}\label{eq:above-guard-scaling}
 D_{F_{M,D}}(x,y)\ge
 \frac12\|\nabla F_{M,D}(x)-\nabla F_{M,D}(y)\|_q^2
 \quad\Longleftrightarrow\quad
 \mathsf C_M(X,Y).
\end{equation}
Thus every inequality used by the two finite energy arguments is observable
from consecutive cached pairs.

If $D\ge R$, a closest minimizer $x^*$ maps to
$y^*=(x^*-x_0)/D$ with
\begin{equation}\label{eq:above-radius-entry}
 \norm{y^*}_p\le1,
 \qquad h(y^*)\le\frac1p.
\end{equation}
This implication is used only in the analysis; the routine neither knows
$R$ nor queries an optimizer.

\paragraph{Phase I.}
Run \eqref{eq:above-primal-s}--\eqref{eq:above-primal-x} on
$F_{M,D}$ with horizon $N_F$ from \eqref{eq:above-NF}.  The pair at $x_0$
is cached.  After constructing each new normalized point $x_{k+1}^{\rm P}$,
query the physical point $X_{k+1}=x_0+Dx_{k+1}^{\rm P}$.  If its returned
gradient is at most $\eps$, return $(\Success,X_{k+1})$.  Otherwise check
\begin{equation}\label{eq:above-primal-guard}
 \mathsf C_M(X_k,X_{k+1}).
\end{equation}
A failed check returns $\Scale$ together with the failed numerical
inequality.  If all $N_F$ checks pass, denote the final queried point by
$Q_0=X_{N_F}$.

\paragraph{Phase II.}
Recenter without a new query:
\begin{equation}\label{eq:above-phase-two-normalization}
 F_2(y):=\frac{f(Q_0+Dy)-f(Q_0)}{MD^2}.
\end{equation}
Set $q_0^{\rm D}=0$ and use the cached pair at $Q_0$ as the initial pair,
so $G_0=\nabla F_2(0)$.  Run
\eqref{eq:above-dual-init}--\eqref{eq:above-dual-r} with horizon $N_D$
from \eqref{eq:above-ND}.  For each update, first form $q_{k+1}^{\rm D}$
from cached gradients, then query
$Q_{k+1}=Q_0+Dq_{k+1}^{\rm D}$.  Return
$(\Success,Q_{k+1})$ immediately if its gradient is at most $\eps$;
otherwise check
\begin{equation}\label{eq:above-dual-guard}
 \mathsf C_M(Q_k,Q_{k+1}).
\end{equation}
A failed check returns $\Scale$.  If all checks pass, the terminal point
$Q_{N_D}$ has already been queried.  Return $\Success$ if its gradient is at
most $\eps$ and $\Radius$ otherwise.

Both guard families use cached exact values and gradients, so they add no
oracle calls.  The only nonlinear internal operation is the coordinatewise
map \eqref{eq:above-mirror-map}; the routine uses no value of $f^*$, inverse
gradient, or optimization oracle.

\begin{proposition}[The finite $p>2$ trial]
\label{prop:abovetrial}
The routine above has the following three outcomes:
\begin{enumerate}[label=(\roman*)]
\item $\Success$ returns a queried point with gradient norm at most $\eps$;
\item $\Scale$ accompanies a failed consecutive-pair $\mathsf C_M$ test and
hence implies $M<L$;
\item $\Radius$ is possible only after all tests pass and implies $D<R$.
\end{enumerate}
Its pair-oracle cost satisfies
\begin{equation}\label{eq:abovetrialcost}
 N_{\rm trial}\le C_p\kappa^{p/(p+2)}.
\end{equation}
\end{proposition}

\begin{proof}
The $\Success$ and $\Scale$ implications follow directly from the returned
pair and \eqref{eq:cocoercivity-scale}.  To prove the radius implication,
first assume $D\ge R$ and that all guards pass.  By
\eqref{eq:above-radius-entry}, the Phase-I minimizer satisfies the hypothesis
of \eqref{eq:above-phase-one-gap}.  The accepted guards are exactly the
normalized remainders used in \eqref{eq:above-Ustep}, so
\eqref{eq:above-Delta} applies and gives
\begin{equation}\label{eq:above-physical-gap}
 f(Q_0)-f^*\le MD^2\delta=\eps D.
\end{equation}
This gap is an analytical consequence, not a stopping test.

For the recentered function $F_2$, \eqref{eq:above-physical-gap} says
$F_2(0)-\inf F_2\le\delta$.  The Phase-II guards are exactly the normalized
remainders in \eqref{eq:above-Vstep}; convexity supplies the other function
Bregman terms, and \eqref{eq:above-oriented-remainder} supplies the mirror
terms.  Hence \eqref{eq:above-normalized-terminal} gives
\[
 \norm{\nabla f(Q_{N_D})}_q
 =MD\norm{\nabla F_2(q_{N_D}^{\rm D})}_q
 \le MD\delta=\eps.
\]
We have proved the forward implication
\begin{equation}\label{eq:above-forward-success}
 D\ge R\ \text{ and all Phase-I/II guards pass}
 \quad\Longrightarrow\quad \Success.
\end{equation}
The exhaustive-outcome contrapositive shows that a completed nonsuccessful
trial has $D<R$, which is the implication attached to $\Radius$.

Finally, \eqref{eq:above-NF} and \eqref{eq:above-ND} give
$N_F+N_D=O_p(\delta^{-p/(p+2)})$.  Phase~I reuses the cached pair at $x_0$;
Phase~II reuses the Phase-I endpoint.  It therefore makes at most
$N_F+N_D$ new queries, and early termination only decreases this number.
Since $\delta^{-1}=\kappa>1$, integer ceilings and a fixed initialization
cost are absorbed into \eqref{eq:abovetrialcost}.
\end{proof}

For the upper clause of \Cref{prop:pgtwo-optimality}, first query $x_0$ as
required by the oracle model and put $G=\|\nabla f(x_0)\|_q$.  If
$G\le\eps$, return this charged query.  Otherwise smoothness and the radius
bound give $\eps<G\le LR$, so $\kappa=LR/\eps>1$.  Take $M=L$ and
$D=R$.  Then $D\ge G/M$, and every
$\mathsf C_L$ test passes by \eqref{eq:truecocoercivity}.  Restoring the
physical scale in \eqref{eq:above-trial-normalization} gives the call count
$O_p(1+(LR/\eps)^{p/(p+2)})$ and a queried output.  This also shows why the
known-parameter rate and the parameter-free local trial share the same
finite algorithm.

\subsection{The deterministic lower bound}

The hard family begins with the local smoothing and resisting completion of
Guzm\'an--Nemirovski and combines it with the radial cap used by
Diakonikolas--Guzm\'an for unrestricted-query lower bounds
\citep[Secs.~2.1--2.3, Props.~1--2]{GuzmanNemirovski2015}
\citep[Sec.~2.1, Eq.~(3), observations O1--O2]{DiakonikolasGuzman2020}.
These constructions provide the smoothing kernel, its locality and curvature
estimates, the adversarial maximum, and the radial cap.  We give the exact
value--gradient completion, the gradient bound at every unrestricted query,
and the completion-independent radius normalization explicitly.

Let $d\ge2$, let $1\le T\le d$, and put
\begin{equation}\label{eq:above-lower-r0}
 r_0:=\min\{p,3\ln d\}>2.
\end{equation}
Choose $\theta>1$ sufficiently close to one that $2\theta<r_0$, and define
\begin{equation}\label{eq:above-lower-kernel}
 \varphi(x):=2\left(\sum_{j=1}^d|x_j|^{r_0}\right)^{2\theta/r_0},
 \qquad
 G_p:=\{x:\norm{x}_p\le1\}.
\end{equation}
The function $\varphi$ is convex, $C^2$, invariant under signed coordinate
permutations, and satisfies
$\varphi(0)=0$, $\nabla\varphi(0)=0$, and
$\varphi(x)>\|x\|_p$ on $\partial G_p$.  Differentiation and H\"older's
inequality give
\begin{equation}\label{eq:above-lower-Hessian}
 \langle e,\nabla^2\varphi(x)e\rangle
 \le4\theta(r_0-1)\norm{x}_{r_0}^{2\theta-2}\norm{e}_{r_0}^2.
\end{equation}
If $p\le3\ln d$, take $r_0=p$ and $\theta$ close enough to one to obtain
$5p\|e\|_p^2$ on $G_p$.  If $p>3\ln d$, use
\[
 \norm{e}_{r_0}\le d^{1/r_0-1/p}\norm e_p
 \le e^{1/3}\norm e_p
\]
and again take $\theta$ close enough to one.  Thus this kernel meets
all smoothing assumptions with
\begin{equation}\label{eq:above-lower-Mpd}
 M_{p,d}\le
 \begin{cases}
 5p,&p\le3\ln d,\\
 15e^{2/3}\ln d,&p>3\ln d,
 \end{cases}
 \qquad
 M_{p,d}\le C\min\{p,\ln d\}.
\end{equation}

For a convex one-Lipschitz function $\ell$ with respect to $\|\cdot\|_p$
and $\chi>0$, define the local smoothing
\begin{equation}\label{eq:above-lower-smoothing}
 \mathcal S_\chi[\ell](x)
 :=\min_v\{\ell(x+v)+\chi\varphi(v/\chi)\}.
\end{equation}
The boundary inequality for $\varphi$ places every minimizing displacement
in $\chi\operatorname{int}G_p$: on the boundary, Lipschitzness makes the
objective strictly larger than its value at $v=0$.  The zero displacement
and $\varphi\ge0$ then give
\begin{equation}\label{eq:above-lower-smoothing-approx}
 \ell(x)-\chi\le\mathcal S_\chi[\ell](x)\le\ell(x).
\end{equation}
The first-order condition, monotonicity of $\partial\ell$, and
\eqref{eq:above-lower-Hessian} give
\begin{equation}\label{eq:above-lower-smoothing-smooth}
 \norm{\nabla\mathcal S_\chi[\ell](x)
 -\nabla\mathcal S_\chi[\ell](y)}_q
 \le\frac{M_{p,d}}\chi\norm{x-y}_p.
\end{equation}
The value and derivative at $x$ depend only on $\ell$ in $x+\chi G_p$, and
for a signed coordinate permutation $Q$,
\begin{equation}\label{eq:above-lower-equivariance}
 \mathcal S_\chi[\ell\circ Q](x)=\mathcal S_\chi[\ell](Qx).
\end{equation}
These are the exact approximation, smoothness, locality, and symmetry
properties used below.

Translate the first query of an arbitrary deterministic algorithm to the
origin and set
\begin{equation}\label{eq:above-lower-parameters}
 \Delta_T:=T^{-1/p},
 \qquad
 \delta_T:=\frac{\Delta_T}{2T},
 \qquad
 \chi_T:=\frac{\delta_T}{2}=\frac{\Delta_T}{4T},
 \qquad
 \beta_T:=\frac{\chi_T}{M_{p,d}}.
\end{equation}
After the algorithm chooses its $t$th query $z_t$, select an unused
coordinate
\begin{equation}\label{eq:above-lower-coordinate}
 \sigma(t)\in\arg\max_{j\in\{1,\ldots,T\}
 \setminus\{\sigma(1),\ldots,\sigma(t-1)\}}|(z_t)_j|
\end{equation}
and choose $\xi_t\in\{-1,1\}$ so that
$\xi_t(z_t)_{\sigma(t)}=|(z_t)_{\sigma(t)}|$.  Define
\begin{align}
 g_t(x)&:=\max_{1\le i\le t}
 \{\xi_i x_{\sigma(i)}-(i-1)\delta_T\},
 \label{eq:above-lower-partial-g}\\
 H_t(x)&:=\max\left\{\frac12g_t(x),\norm{x}_p-\frac32\right\},
 \label{eq:above-lower-partial-H}\\
 \bar f_t(x)&:=\beta_T\mathcal S_{\chi_T}[H_t](x).
 \label{eq:above-lower-partial-f}
\end{align}
Return the exact pair of $\bar f_t$ at $z_t$.  After $T$ rounds, write
$g=g_T$, $H=H_T$, and $\bar f=\bar f_T$.  Both branches of $H_t$ are convex
and one-Lipschitz, so \eqref{eq:above-lower-smoothing-smooth} and
\eqref{eq:above-lower-parameters} make every $\bar f_t$, including the
completion $\bar f$, convex and one-smooth from $\ell_p$ to $\ell_q$.

\begin{lemma}[Exact pair completion]
\label{lem:above-lower-completion}
The completed function $\bar f$ returns at $z_1,\ldots,z_T$ exactly the
value--gradient pairs used to generate the deterministic transcript.
\end{lemma}

\begin{proof}
Fix $t<s$.  The coordinate $\sigma(s)$ was unused at time $t$, so maximality
in \eqref{eq:above-lower-coordinate} gives
\begin{equation}\label{eq:above-lower-future-piece}
 \xi_s(z_t)_{\sigma(s)}
 \le |(z_t)_{\sigma(s)}|
 \le |(z_t)_{\sigma(t)}|
 =\xi_t(z_t)_{\sigma(t)}.
\end{equation}
After the offsets, the $t$th affine piece exceeds the $s$th at $z_t$ by at
least $(s-t)\delta_T\ge\delta_T$.  The difference of two pieces is
$2$-Lipschitz.  Since $\chi_T=\delta_T/2$, no future piece exceeds the
$t$th piece anywhere in the complete neighborhood $z_t+\chi_TG_p$; equality
at its boundary leaves the maximum unchanged.  Hence $g$ and $g_t$, and
therefore $H$ and $H_t$, agree throughout that neighborhood.  Locality
of \eqref{eq:above-lower-smoothing} makes $\bar f$ and $\bar f_t$ agree near
$z_t$, so their exact values and gradients coincide.  Induction over the
transcript proves the claim.
\end{proof}

Thus the lemma completes both values and gradients for every unrestricted
query made by the deterministic algorithm.

\begin{lemma}[A gap at every query]
\label{lem:above-lower-gap}
The function $\bar f$ is coercive, attains its minimum, and
\begin{equation}\label{eq:above-lower-query-gap}
 \bar f(z_t)-\min_x\bar f(x)
 \ge\frac{\beta_T\Delta_T}{4}
 =\frac{1}{16M_{p,d}T^{1+2/p}}
 \qquad(1\le t\le T).
\end{equation}
\end{lemma}

\begin{proof}
The affine piece selected at time $t$ gives
$g(z_t)\ge-(t-1)\delta_T$, so
\begin{equation}\label{eq:above-lower-query-value}
 \mathcal S_{\chi_T}[H](z_t)
 \ge H(z_t)-\chi_T
 \ge-\frac{(T-1)\delta_T}{2}-\chi_T
 =-\frac{\Delta_T}{4}.
\end{equation}
Define $u$ by
\begin{equation}\label{eq:above-lower-witness}
 u_{\sigma(i)}=-\xi_i\Delta_T\quad(1\le i\le T),
 \qquad u_j=0\quad\text{otherwise}.
\end{equation}
Then $\|u\|_p=1$, $g(u)=-\Delta_T$, and
$H(u)=-\Delta_T/2$.  Taking zero smoothing displacement gives
$\mathcal S_{\chi_T}[H](u)\le-\Delta_T/2$.  Combining this with
\eqref{eq:above-lower-query-value} proves the gap.  Finally,
$H(x)\ge\|x\|_p-3/2$, so
\eqref{eq:above-lower-smoothing-approx} gives
$\bar f(x)\ge\beta_T(\|x\|_p-3/2-\chi_T)$, which proves coercivity and
attainment.
\end{proof}

The radial cap in \eqref{eq:above-lower-partial-H} is the one used by
Diakonikolas--Guzm\'an \citep[Sec.~2.1, Eq.~(3)]{DiakonikolasGuzman2020}.
The next calculation verifies the entire-neighborhood identity needed for a
queried-gradient lower bound.

\begin{lemma}[The gradient outside the fixed ball]
\label{lem:above-lower-outside}
If $\norm{x}_p\ge4$, then
\begin{equation}\label{eq:above-lower-outside-gradient}
 \norm{\nabla\mathcal S_{\chi_T}[H](x)}_q=1,
 \qquad
 \norm{\nabla\bar f(x)}_q=\beta_T.
\end{equation}
Every minimizer of $\bar f$ therefore lies in $\{x:\|x\|_p<4\}$.
\end{lemma}

\begin{proof}
Each affine piece in $g$ has dual norm one and a nonpositive offset, so
\begin{equation}\label{eq:above-lower-radial-domination}
 \frac12g(y)\le\frac12\norm y_p.
\end{equation}
If $\|x\|_p\ge4$ and $v\in\chi_TG_p$, then
$\|x+v\|_p\ge4-\chi_T>3$.  Hence the radial branch dominates throughout the
smoothing neighborhood, and
\begin{equation}\label{eq:above-lower-radial-envelope}
 \mathcal S_{\chi_T}[H](x)
 =\min_{v\in\chi_TG_p}
 \left\{\norm{x+v}_p-\frac32+\chi_T\varphi(v/\chi_T)\right\}.
\end{equation}
At a minimizing displacement $v_x$, the envelope identity gives
\begin{equation}\label{eq:above-lower-envelope-gradient}
 \nabla\mathcal S_{\chi_T}[H](x)=\nabla\norm{x+v_x}_p.
\end{equation}
The envelope is differentiable, so this gradient is independent of the
choice of a minimizing displacement.  Since $x+v_x\ne0$,
\begin{equation}\label{eq:above-lower-norm-gradient}
 [\nabla\norm w_p]_j
 =\frac{\operatorname{sgn}(w_j)|w_j|^{p-1}}{\norm w_p^{p-1}},
 \qquad
 \norm{\nabla\norm w_p}_q=1.
\end{equation}
This proves \eqref{eq:above-lower-outside-gradient}.  A minimizer of a
differentiable convex function has zero gradient, proving the final claim.
\end{proof}

\begin{proposition}[Unit-smooth queried-gradient bound]
\label{prop:above-lower-base-gradient}
Every first-$T$ query of the completed function satisfies
\begin{equation}\label{eq:above-lower-base-gradient}
 \norm{\nabla\bar f(z_t)}_q
 \ge\frac{1}{128M_{p,d}T^{1+2/p}}.
\end{equation}
\end{proposition}

\begin{proof}
If $\|z_t\|_p\ge4$, \Cref{lem:above-lower-outside} gives
$\|\nabla\bar f(z_t)\|_q=\beta_T$, which is larger than the right side of
\eqref{eq:above-lower-base-gradient}.  Otherwise choose an actual minimizer
$z^*$ of $\bar f$.  The same lemma gives $\|z^*\|_p<4$, and hence
$\|z_t-z^*\|_p<8$.  Convexity and \Cref{lem:above-lower-gap} now give
\begin{align}
 \frac{\beta_T\Delta_T}{4}
 &\le\bar f(z_t)-\bar f(z^*)
 \le\langle\nabla\bar f(z_t),z_t-z^*\rangle\nonumber\\
 &\le8\norm{\nabla\bar f(z_t)}_q.
 \label{eq:above-lower-inside-gradient}
\end{align}
Since $\beta_T=\Delta_T/(4TM_{p,d})$ and
$\Delta_T^2=T^{-2/p}$, this is exactly
\eqref{eq:above-lower-base-gradient}.  The inside and outside cases cover
every possible query.
\end{proof}

It remains to fix the actual supplied radius without allowing the scale to
depend on unseen choices.  We combine the signed-permutation-invariant kernel
of Guzm\'an--Nemirovski and the radial family of
Diakonikolas--Guzm\'an with a completion-independent scaling
\citep{GuzmanNemirovski2015,DiakonikolasGuzman2020}.

\begin{lemma}[Completion-independent optimizer radius]
\label{lem:above-lower-radius}
There is a number $r_T$, depending only on $p,d,T$ and the fixed kernel, but
not on $(\sigma,\xi)$, such that
\begin{equation}\label{eq:above-lower-radius-range}
 r_T=\operatorname{dist}_p(0,\arg\min\bar f),
 \qquad \frac14<r_T<4.
\end{equation}
\end{lemma}

\begin{proof}
For a completion $(\sigma,\xi)$, extend
$(Qx)_i=\xi_i x_{\sigma(i)}$, $1\le i\le T$, to a signed coordinate
permutation.  Under $Q$, every completed maximum becomes the same canonical
maximum with offsets indexed by $i$.  The $p$-norm, $G_p$, and $\varphi$ are
invariant, so \eqref{eq:above-lower-equivariance} makes all completed
functions isometric images of one canonical function.  Their distances from
the origin to their minimizer sets are therefore equal; call this common
distance $r_T$.  \Cref{lem:above-lower-outside} gives $r_T<4$.

For the lower bound, the maximum of the affine pieces is at least their
average.  The dual norm of
$T^{-1}\sum_{i=1}^T\xi_i e_{\sigma(i)}$ is
$T^{-1/p}=\Delta_T$, so
\begin{equation}\label{eq:above-lower-average}
 g(x)\ge-\Delta_T\norm{x}_p-\frac{(T-1)\delta_T}{2}.
\end{equation}
If $\|x\|_p\le1/4$, then
\begin{align}
 \mathcal S_{\chi_T}[H](x)
 &\ge H(x)-\chi_T\nonumber\\
 &\ge-\frac{\Delta_T}{8}-\frac{(T-1)\delta_T}{4}-\chi_T
 =-\frac{2T+1}{8T}\Delta_T
 \ge-\frac{3\Delta_T}{8}.
 \label{eq:above-lower-small-ball}
\end{align}
The witness in \eqref{eq:above-lower-witness} has smoothed value at most
$-\Delta_T/2$.  Thus no point in the closed ball of radius $1/4$ minimizes
$\bar f$, proving $r_T>1/4$.
\end{proof}

Return now to the arbitrary physical deterministic algorithm.  Translate
its first query to zero and define the common change of variables
\begin{equation}\label{eq:above-lower-physical-function}
 f_{\rm hard}(x)
 :=\frac{LR^2}{r_T^2}\,
 \bar f\!\left(\frac{r_T}{R}(x-x_0)\right).
\end{equation}
Writing $z=(r_T/R)(x-x_0)$ gives
\begin{equation}\label{eq:above-lower-physical-gradient}
 \nabla f_{\rm hard}(x)=\frac{LR}{r_T}\nabla\bar f(z).
\end{equation}
Unit smoothness of $\bar f$ makes $f_{\rm hard}$ globally $L$-smooth from
$\ell_p$ to $\ell_q$, and its minimizer set has actual distance exactly $R$
from $x_0$.  The factor $r_T$ is common to every completion, so the scaled
oracle transcript cannot reveal future coordinates or signs.  By
\Cref{lem:above-lower-completion,prop:above-lower-base-gradient} and
$r_T<4$, every first-$T$ physical query obeys
\begin{equation}\label{eq:above-lower-final-gradient}
 \norm{\nabla f_{\rm hard}(x_t)}_q
 \ge\frac{LR}{512M_{p,d}T^{1+2/p}}.
\end{equation}
Solving \eqref{eq:above-lower-final-gradient} for $T$ gives
\eqref{eq:pgtwo-known-lower-count}, provided the selected integer horizon
satisfies $T\le d$.  Together with the known-parameter upper construction,
this completes the proof of \Cref{prop:pgtwo-optimality}.

The exponent and the factor $\min\{p,\ln d\}$ are those of the
Diakonikolas--Guzm\'an lower benchmark
\citep[Thm.~4 and Cor.~2]{DiakonikolasGuzman2024}.  The matching statement is
therefore a statement about the deterministic polynomial exponent for fixed
finite $p$, with the displayed dimension qualification.

\section{Geometric summation and completion of the parameter-free theorem}
\label{sec:global}

The global argument separates three parts.  Within a fixed scale epoch,
returns showing that the radius guess is too small generate a geometric
sequence of trial radii.  Across epochs, returns showing that the scale guess
is too small generate a second geometric sequence of terminal costs.  The
initial calibration occurs before both sequences and contributes
an additive logarithm.  The proof below carries out these two sums separately,
then proves finite termination and
queried-output correctness independently of the cost calculation.

The three branches now provide exactly the outcomes required by
\Cref{prop:certification}.  With $\kappa=MD/\eps\ge1$, their trial costs are
\begin{equation}\label{eq:trialcostabstract}
 T_p(\kappa)\le
 \begin{cases}
 C_p\kappa^{1/2},&1<p<2,\\
 C\kappa^{1/2},&p=2,\\
 C_p\kappa^{\alpha_p},&p>2,
 \end{cases}
 \qquad \alpha_p:=\frac{p}{p+2}.
\end{equation}
In the nontrivial controller branch, $\kappa\ge G/\eps>1$, so the additive
two in \eqref{eq:below-local-count} is absorbed into its first line; it does
not accumulate once per trial.
The first line is \Cref{prop:belowtrial}, the second is
\Cref{prop:euclideantrial}, and the third is \Cref{prop:abovetrial}.
Every scale return contains an actually failed observable inequality; every
radius return is the contrapositive of a fully guarded
known-radius proof.
The radius-then-scale accounting follows the parameter-adaptation procedure
of Lan--Ouyang--Zhang \citep[Sec.~4]{LanOuyangZhang2026}; the
unknown-distance adaptive regularization of Ito and Fukuda is a related
comparison \citep[Thm.~4.2]{ItoFukuda2021}.  The above-two routine has the
same three outcomes, so the two sums below apply with exponent $p/(p+2)$.

\begin{lemma}[Geometric summation of trial costs]\label{lem:amortization}
Let $M_s=2^sM_{\rm a}$ be the scale in epoch $s$ and
\[
 D_{s,j}=2^j\frac{G}{M_s},\qquad j=0,1,\ldots,
\]
until the epoch returns $\Scale$ or $\Success$.  Then all visited scales obey
$M_s<2L$, all expensive radii obey $D_{s,j}\le2R$, and summing
\eqref{eq:trialcostabstract} over the realized path gives the corresponding
main term in \eqref{eq:mainbelow}, \eqref{eq:maineuclidean}, or
\eqref{eq:mainabove}.
\end{lemma}

\paragraph{Proof roadmap.}
The proof has two distinct summation stages.  First fix an epoch $s$ and use
the preceding $\Radius$ return to bound its last dimensionless radius
$\kappa_{s,j}$ by $2k_s$; the powers of the doubling sequence then sum to the
one-epoch bound.  Next use $k_{s+1}=2k_s$ and the final-scale bound
$M_s<2L$ to sum the one-epoch bounds over $s$.  We perform both stages for all
three regimes using only geometric power sums.

\begin{proof}
The scale and radius bounds were established in
\Cref{prop:certification}: every $M_s<2L$ and every expensive
$D_{s,j}\le2R$.  We retain their exact normalization for the count.  Set
\[
 \beta=\frac G\eps>1,\qquad
 k_s=\frac{M_sR}{\eps},\qquad
 \kappa_{s,j}=\frac{M_sD_{s,j}}\eps=2^j\beta.
\]
The calibration lemma gives $\beta\le2k_0$.  For $j\ge1$, the preceding radius was
certified to be below $R$, and hence
\begin{equation}\label{eq:kappabound}
 \kappa_{s,j}<2k_s.
\end{equation}

\emph{Stage 1: radius trials inside one epoch.}
For $j=0$, the calibration lemma gives
\[
 \beta=\frac G\eps
 =\frac{M_{\rm a}D_{\rm a}}\eps
 \le\frac{2M_{\rm a}R}{\eps}=2k_0\le2k_s.
\]
For $j\ge1$, the preceding radius $D_{s,j-1}$ returned $\Radius$ and is
strictly below $R$; hence
\[
 \kappa_{s,j}
 =\frac{M_s(2D_{s,j-1})}\eps<2k_s,
\]
which is \eqref{eq:kappabound}.  Moreover
$\kappa_{s,j}=2^j\beta$.  For any exponent $a>0$ and terminal index $J_s$,
\begin{equation}\label{eq:geometricpowersum}
 \sum_{j=0}^{J_s}\kappa_{s,j}^a
 =\kappa_{s,J_s}^a\sum_{r=0}^{J_s}2^{-ar}
 \le\frac{\kappa_{s,J_s}^a}{1-2^{-a}}.
\end{equation}
Thus every power cost is terminal-dominated.  Applying
\eqref{eq:geometricpowersum} with $a=1/2$ below and at two, or with
$a=\alpha_p$ above two, yields the complete one-epoch bound
\begin{equation}\label{eq:one-epoch}
 \sum_jT_p(\kappa_{s,j})\le
 \begin{cases}
 C_pk_s^{1/2},&1<p<2,\\
 Ck_s^{1/2},&p=2,\\
 C_pk_s^{\alpha_p},&p>2.
 \end{cases}
\end{equation}

\emph{Stage 2: scale epochs.}
Every epoch-ending $\Scale$ outcome doubles $M_s$, so
\[
 k_{s+1}=\frac{(2M_s)R}{\eps}=2k_s.
\]
If $S_*$ is the terminal epoch, the scale invariant gives
\begin{equation}\label{eq:terminalepochcondition}
 k_{S_*}=\frac{M_{S_*}R}{\eps}<\frac{2LR}{\eps}=2K.
\end{equation}
Apply the same backward geometric sum to the powers of $k_s$.  For either
$1<p<2$ or $p=2$, with the appropriate constant, this gives
\[
 \sum_{s=0}^{S_*}C_pk_s^{1/2}
 \le C_pk_{S_*}^{1/2}\sum_{r=0}^{S_*}2^{-r/2}
 \le C_pK^{1/2}.
\]
For $p=2$, $C_p$ here is replaced by the universal constant $C$.
For $p>2$, replacing $1/2$ by $\alpha_p$ gives
$C_pK^{\alpha_p}$.
These are exactly the three main terms in \Cref{thm:main}.  In the nontrivial branch
$G>\eps$, \eqref{eq:Gbound} implies $K>1$; replacing $K$ by $\Kbar$ also
covers the immediate-success branch of the theorem.
\end{proof}

The first sum prices radius discovery at a fixed curvature scale; the second
prices curvature discovery itself.  The two failure implications permit this
ordered accounting.  A rectangular search would
have no corresponding decomposition over the realized sequence of trials.

It remains to account for immediate success and the initial calibration, prove
finite termination and queried-output correctness, and verify that every
implemented quantity belongs to $(p,d,\eps,x_0,z_0,M_0)$ or to the observed
transcript.

\begin{proof}[Completion of the proof of \Cref{thm:main}]
Query $x_0$.  If $G\le\eps$, return $x_0$ after one post-initialization call;
the returned point has just been queried and this constant cost is covered by
$\Kbar\ge1$.  Henceforth $G>\eps$.  Then \eqref{eq:Gbound} gives
$\eps<G\le LR$, so $K>1$.

\emph{Initialization overhead.}
Run the initial calibration in \Cref{lem:anchor}.  It uses at
most
\[
 1+\left\lceil\log_2\frac{L}{M_0}\right\rceil
 \le C\log\!\left(e+\frac{L}{M_0}\right)
\]
pair queries and returns $M_{\rm a}<2L$ and $D_{\rm a}\le2R$.

\emph{Finite termination.}
Use \Cref{alg:controller} with the local trial appropriate to the supplied
fixed exponent $p$.  Each local trial is finite.  A $\Scale$ outcome can occur
only while $M<L$ and then doubles $M$; consequently only finitely many scale
epochs occur before one has $M\ge L$.  Within any fixed epoch, a $\Radius$
outcome can occur only while $D<R$ and then doubles $D$; consequently only
finitely many radius outcomes occur before either the epoch returns $\Scale$
or the current radius reaches $R$.  At an epoch with $M\ge L$, no scale guard
can fail.  Once also $D\ge R$, the appropriate local proposition forces
$\Success$.  Therefore the controller terminates after finitely many trials.

\emph{Queried-output correctness.}
The controller returns only on a $\Success$ outcome.  In all three local
propositions, success is declared only after evaluating the pair oracle at the
reported terminal point and checking its returned gradient.  Thus the output
$\widehat x$ is actually queried and satisfies
$\|\grad f(\widehat x)\|_q\le\eps$.

\emph{Cost.}
The below-two, Euclidean, and above-two trial bounds are exactly the three
lines of \eqref{eq:trialcostabstract}.  \Cref{lem:amortization} first sums
radius trials within an epoch and then scale epochs, giving respectively
\[
 C_pK^{1/2},\qquad CK^{1/2},\qquad
 C_pK^{p/(p+2)}.
\]
Adding the initialization bound gives the second term in
\eqref{eq:mainbelow}--\eqref{eq:mainabove}; replacing $K$ by $\Kbar$ covers
the immediate-success case.

\emph{Information and dimension.}
The implementation uses only $p,d,\eps,x_0,z_0,M_0$, deterministic internal
minimizations of explicit models, and queried values and gradients.  The
quantities $L,R,f^*$ and an optimizer appear only in the proof of the two
failure implications.  Every geometric constant used above depends only on
fixed $p$
or is universal; no norm conversion introducing a polynomial factor in $d$
occurs.  This establishes every clause of \Cref{thm:main}.
\end{proof}

The three local routines cover $1<p<2$, the Euclidean case, and $p>2$; their
costs are combined by the two geometric sums above.

\section{Conclusion}
\label{sec:conclusion}

We studied parameter-free gradient minimization for smooth convex objectives over $\ell_p$ geometries, with a particular focus on what information is needed to obtain a finite guarantee when the smoothness and distance scales are unknown.  In the strict counted local value-gradient model, a finite complexity bound depending only on $LR/\eps$ is impossible without some nondegenerate scale information, even in one dimension and even when $LR/\eps=4$.  With a nondegenerate secant observation supplied at initialization, this obstruction disappears.  The resulting method applies for every fixed $1<p<\infty$ and returns a queried point with gradient norm at most $\eps$.  Its post-initialization oracle complexity is of order
$\Kbar^{1/2}$ for $1<p\le2$ and $\Kbar^{p/(p+2)}$ for $2<p<\infty$, up to the additive logarithmic cost of the initial scale calibration.

The construction has two main components.  The initial calibration produces working estimates of the unknown smoothness and radius scales from the supplied secant information.  The subsequent controller increases a scale or radius estimate only after the corresponding observable test has failed. This allows the local routines for $1<p<2$, $p=2$, and $p>2$ to be combined without introducing an additional multiplicative logarithm in the main polynomial term.  In particular, for finite $p>2$ the local construction attains the exponent $p/(p+2)$, and the parameter-free controller preserves this exponent when $L$ and $R$ are not given to the method.

\begin{small}
\paragraph{Acknowledgement.} The authors used Chatgpt 5.6 for assistance with proof exploration, verification, and manuscript preparation. The mathematical statements and proofs in this work were also formalized and mechanically verified in Lean    available at \url{https://github.com/yuningyang19/parameter_free_gradient_p}.  All mathematical content was independently reviewed by the authors.
\end{small}

\bibliographystyle{plainnat}
\bibliography{references}

\end{document}